\documentclass[twoside]{article}

\usepackage[margin=1in]{geometry}

\title{Notes on Statistically Invariant States in Stochastically Driven Fluid Flows}
\author{Nathan Glatt-Holtz\\
  \scriptsize{Department of Mathematics, Virginia Tech}\\
  \scriptsize{email: negh@math.vt.edu}}
\date{}

\usepackage{amsfonts, amssymb, amsmath, amsthm}
\usepackage[colorlinks=true, pdfstartview=FitV, linkcolor=blue, 
            citecolor=blue, urlcolor=blue]{hyperref}

\numberwithin{equation}{section}

\newtheorem{Thm}{Theorem}[section]
\newtheorem{Lem}{Lemma}[section]
\newtheorem{Prop}{Proposition}[section]
\newtheorem{Cor}{Corollary}[section]

\newtheorem{Def}{Definition}[section]
\newtheorem{Rmk}{Remark}[section]
\newtheorem{Asmp}{Assumption}[section]

\newcommand{\indFn}[1]{1 \! \! 1_{#1}}
\newcommand{\E}{\mathbb{E}}
\newcommand{\Prb}{\mathbb{P}}

\newcommand{\RR}{\mathbb{R}}
\newcommand{\SN}{\mathbb{S}^N}
\newcommand{\JJ}{\mathcal{J}}
\newcommand{\AAA}{\mathcal{A}}
\newcommand{\MM}{\mathcal{M}}

\newcommand{\MD}{\mathfrak{D}}
\newcommand{\DD}{\mathbb{D}}

\newcommand{\dir}{\xi}
\newcommand{\Dtest}{\mathcal{T}}

\newcommand{\expCon}{\kappa}

\begin{document}
\markboth{Nathan Glatt-Holtz}
{Notes on Statistically Invariant States in Stochastically Driven Fluid Flows}

\maketitle

\begin{abstract}
These expository notes address certain stationary and ergodic properties of the
equations of fluid dynamics subject to a spatially degenerate (i.e. frequency localized), 
white in time gaussian forcing.  In order to provide
an accessible treatment of some recent progress in
this subject  (cf. \cite{HairerMattingly06, HairerMattingly2011, FoldesGlattHoltzRichardsThomann2013}) 
we will develop ideas in detail for a class of finite dimensional models.
\hfill DRAFT: \today.
\end{abstract}

\setcounter{tocdepth}{1}
\tableofcontents

\section*{Introduction}

Going back at least to \cite{Novikov1965}, stochastic forcing
has been used to model the large scale stirring driving turbulent
fluids.    One might thus hope that invariant measures for this class 
of stochastic PDEs would contain statistics predicted by theories of turbulent
flow.  As such, an ongoing mathematical challenge has been
to characterize these statistically stationary states, their attraction
properties and to obtain various asymptotics in this class of measures
for relevant physical parameters.

The use of a gaussian, white in time and spatially correlated (frequency
localized) forcing is attractive for several reasons.  Firstly the 
classical theories of turbulence posit that statistics,
for example the famous $k^{-5/3}$ cascade of Kolmogorov 
should be observed in the inertial range independent of the specific
details of the forcing.  Stochastic perturbation may thus be viewed as a proxy for
a `generic large scale forcing'.   A second advantage specific to 
white noise is that the effect of the stirring driving the system
is more transparent at the level of basic energy balance equations.  
Finally we have that the equations evolving probability distributions 
of the flow are the (degenerate) parabolic system, the Kolmogorov-Fokker-Plank equations,
rather than the first order Louisville equations encountered in
the deterministic case.  Thus one has a greater hope of obtaining 
uniqueness and ergodicity properties for statistically steady states.

Starting in the mid 90s with the seminal work \cite{FlandoliMaslowski1}
an extensive mathematical literature has developed around the question of uniqueness of invariant measures
for the 2d Navier-Stokes equations perturbed by white noise; see \cite{ZabczykDaPrato1996,
Mattingly1,Mattingly2, BricmontKupiainenLefevere2001, KuksinShirikyan1,KuksinShirikyan2, Mattingly03,MattinglyPardoux1,
HairerMattingly06, Kupiainen10, HairerMattingly2011, Debussche2011a, KuksinShirikian12}.
More recently in a series of groundbreaking works \cite{MattinglyPardoux1, HairerMattingly06, HairerMattingly2008, HairerMattingly2011}
methods to treat the physically important case of a spatially degenerate (that is frequency localized) forcing have been developed.   
These methods have proven to be a useful starting point for the analysis of other
interesting nonlinear stochastic PDEs arising in fluid dynamics,
\cite{ConstantinGlattHoltzVicol2013,FoldesGlattHoltzRichardsThomann2013,FriedlanderGlattHoltzVicol2014}.

These notes aim to provide an accessible and largely self-contained account of some
of the advances in \cite{MattinglyPardoux1, HairerMattingly06, HairerMattingly2008, HairerMattingly2011, FoldesGlattHoltzRichardsThomann2013}.
To simplify the discussion we present the ideas in the context of a class of finite dimensional models. However the analysis will be developed in such a way as to highlight methods that can be readily generalized to an infinite dimensional context.

\section{The Model Equations}
\label{sec:Model}

We will be studying systems of stochastic ordinary differential
equations of the general form
\begin{align}
  dU + (\nu AU + B(U,U)) dt = \sigma dW = \sum_{k=1}^d \sigma_{k} dW^k.
  \label{eq:fin:dim:model}
\end{align}
Here the solution $U: \Omega \times [0, \infty) \to \RR^N$ 
is a stochastic process.  The system is driven by a $d$-dimensional 
standard Brownian motion $W = (W_1, \ldots, W_d)$
where we have in mind that $d << N$.   We assume that $\nu >0$ and
\begin{align}
  A \textrm{ linear and strictly positive definite so that }  \langle \nu A U, U \rangle \geq \alpha |U|^2.
  \label{eq:pos:def:cond}
\end{align}
for every $U \in \RR^N$ (and where $\alpha > 0$ may depend on $\nu$).
$B$ is bilinear and satisfies the cancelation property
\begin{align}  
  \langle B(V, U), U \rangle = 0,
  \label{eq:can:cond}
\end{align}
valid for all $V, U \in \RR^N$.   
Regarding the structure of the stochastic
perturbation we assume an additive, time homogenous noise so that\footnote{Alternatively we may think of $\sigma$ as 
a fixed matrix in $\RR^{d \times N}$.}
\begin{align}
  \sigma_1, \ldots, \sigma_d  \textrm{ are fixed elements in } \RR^N.
  \label{eq:noise:structure}
\end{align}
Finally we will suppose that
\begin{align}
  F \textrm{ and } \sigma_1, \ldots, \sigma_d \textrm{ satisfy the H\"ormander bracket condition }
  \label{eq:Hormander:cond:asmpt}
\end{align}
where $F(U) = \nu A U + B(U,U)$.  We will recall a precise statement of this final
condition, \eqref{eq:Hormander:cond:asmpt} below in Definition~\ref{def:hormander}.

The following theorem concerning \eqref{eq:fin:dim:model} will be the focus of our attention below\footnote{See Section~\ref{sec:markov}
for a recall of precise definition of invariant measures.  Morally speaking Theorem~\ref{thm:unique:ergodic} asserts that \eqref{eq:fin:dim:model} has
a unique statistically steady state and that this distribution represents the time asymptotic statistical behavior of all solutions.}
\begin{Thm}\label{thm:unique:ergodic}
  There exists a unique ergodic invariant measure $\mu$ for \eqref{eq:fin:dim:model} whenever the 
  conditions \eqref{eq:pos:def:cond}--\eqref{eq:Hormander:cond:asmpt} are satisfied.  Under these circumstances $\mu$ 
  is also mixing and satisfies a law of large numbers.
\end{Thm}

\begin{Rmk}
Our model \eqref{eq:fin:dim:model} under the assumptions
\eqref{eq:pos:def:cond}--\eqref{eq:can:cond} cover a variety of 
truncations (discretizations) of the 2 and 3 dimensional Navier-Stokes equations 
as well as a number of other related systems of physical interest.
A very similar class of models has previously been employed as a starting point 
for the study of stochastic fluid flow problems in \cite{AlbeverioFlandoliSinai2008}.  
\end{Rmk}

\begin{Rmk}
We may say that the assumption \eqref{eq:pos:def:cond}--\eqref{eq:can:cond}
are \emph{generic}; alone they do not impose any special structure on the 
way energy is injected into the system or on how these stochastic terms interact with 
the nonlinear terms in $F$.    We will thus derive a number of basic properties
(well-posedness, moment bounds, existence of stationary states) which would be 
expected to hold in a wide variety of settings.  It is in analyzing the uniqueness and attraction properties 
of stationary states that we require a much more detailed understanding of the interaction 
of the nonlinear and stochastic terms as embodied in \eqref{eq:Hormander:cond:asmpt}.
See sections~\ref{sec:our:fav:hor}--\ref{sec:analysis:M} below.
\end{Rmk}

\subsection{Well Posedness}

As usual the starting point is to show that \eqref{eq:fin:dim:model} is a well-posed problem.  Under \eqref{eq:pos:def:cond}--\eqref{eq:can:cond}  we have the following proposition.
\begin{Prop}\label{thm:well:posedness}
Fix a stochastic basis $\mathcal{S} = (\Omega, \mathcal{F}, \Prb, \{\mathcal{F}_t\}_{t \geq 0},
W)$, that is a filtered, right continuous and complete probability space
with $W= ( W_1, \ldots, W_d)$ a $d$-dimensional standard Wiener process relative to this filtration.  
For every $U_0 \in \RR^N$, there exists
a unique $U( \cdot ,U_0): \Omega \times [0, \infty) \to \RR^N$ such that
\begin{align}
  U( \cdot ,U_0) \in C([0, \infty); \RR^N)  \textrm{ a.s. and is } \mathcal{F}_t-\textrm{adapted}
  \label{eq:reg:prop:1}
\end{align} 
and satisfies 
\begin{align}
  U(t) + \int_0^t( \nu AU + B(U,U)) ds = U_0 + \sum_k^d \sigma_k W^k(t)
  \label{eq:int:eq:sol}
\end{align}
for every $t \geq 0$.  Moreover solutions depend continuously on initial conditions, namely,
whenever $U_0^k \to U_0$
\begin{align}
   U(t, U_0^k) \to U(t, U_0),
   \label{eq:cont:dependence}
\end{align}
almost surely for every $t \geq 0$.
\end{Prop}
\begin{proof}
Uniqueness and continuous dependence of solutions on initial 
conditions, \eqref{eq:cont:dependence} is established with a standard Gr\"onwall argument. 

To prove existence of solutions we begin by truncating \eqref{eq:fin:dim:model}
as follows.  Take
\begin{align*}
  dU^R + (\nu AU^R + \theta_R(| U|^2) B(U^R,U^R)) dt = \sum_{k=1}^d \sigma_k dW^k,
\end{align*}
where $\theta_R : \RR \to [0,1]$ is a smooth cut-off function which is identity on $[0, R]$ and zero
on $[R+1, \infty)$.  Observe that this truncated system has globally Lipshitz 
drift terms.  We therefore obtain the global existence of solutions using Banach
fixed point arguments on the space $L^2(\Omega, C([0,T]; \RR^N))$; see e.g. \cite{Oksendal2003}.  Now define
\begin{align*}
  \tau_R  := \inf_{t  \geq 0} \{ |U^R(t)|^2 \geq R\}.
\end{align*}
Up to this stopping time $U^R$ solves \eqref{eq:fin:dim:model}.  By uniqueness
we therefore obtain the existence of solutions up to a possible blow up time $\tau_\infty$.
To see that $\Prb(\tau_\infty < \infty) = 0$ we apply the It\={o} formula and obtain that
\begin{align*}
  d |U|^2 + 2\alpha |U|^2 dt \leq |\sigma|^2 dt + 2\langle U, \sigma\rangle dW
\end{align*}
so that for any $T > 0$, $R > 0$,
\begin{align*}
  \E \left(\sup_{t \in [0, \tau_R \wedge T]} |U(t)|^2\right) \leq C |\sigma|^2 T,
\end{align*}
where the constant $C$ is independent of $R > 0$.
Hence, by the monotone convergence theorem, we conclude that $\tau_\infty  > T$ a.s., completing the proof.

\end{proof}

\section{The Markovian Framework}
\label{sec:markov}
We turn now to recall the Markovian setting for \eqref{eq:fin:dim:model}.  This framework illuminates a deep connection between
SDEs and PDEs and provides the appropriate set-up for us to investigate statistically steady states.  

Take
\begin{align*}
  P_t(U_0, A) = \Prb( U(t, U_0) \in A), \quad t \geq 0, A \in \mathcal{B}(\RR^n)
\end{align*}
where $\mathcal{B}(\RR^N)$ are the Borelian subsets of $\RR^N$.  The \emph{Markov semigroup}
is given by
\begin{align}
  P_t \phi(U_0) = \int_{\RR^N} \phi(U) P_t(U_0, dU) = \E \phi(U(t,U_0))
\end{align}
which acts on $\mathcal{M}_b(\RR^N)$, measurable bounded functions on $\RR^N$.
Intuitively $\{P_t\}_{t \geq 0}$ acts by as evolving forward observable quantities according to the model \eqref{eq:fin:dim:model}
and then averaging the resulting values over many independent realizations.

In view of \eqref{eq:cont:dependence} in Theorem~\ref{thm:well:posedness} we obtain
important continuity properties for $\{P_t\}_{t \geq 0}$.  For example it is easy to see that 
$P_t$ is \emph{Feller}, i.e.
\begin{align*}
  P_t : C_b(\RR^N) \to C_b(\RR^N)
\end{align*}
where $C_b(\RR^N)$ consists of continuous bounded functions and is
stochastically continuous meaning that
\begin{align}
  \lim_{t \to 0} P_t(U_0, B_\delta(U_0)^c) = 0
  \label{eq:stoch:cont}
\end{align}
for any $U_0 \in \RR^N$ and any fixed $\delta > 0$.

To see that $\{P_t\}_{t \geq 0}$ is indeed a semigroup of operators on $C_b(\RR^N)$
we note that solutions $U(\cdot, U_0)$ of \eqref{eq:fin:dim:model} satisfy the Markov property 
namely
\begin{align*}
  \E( \phi(U(t + s, U_0)) | \mathcal{F}_s) = P_t \phi(U(s, U_0))
\end{align*}
almost surely for all $0 < s< t$ and $\phi$ in $C_b(\RR^N)$.\footnote{A rigorous proof of this identity in our setting can be found in \cite{AlbeverioFlandoliSinai2008}.}
Furthermore, $\{P_t\}_{t \geq 0}$ is 
the solution semigroup of a linear partial differential equation
which can be written down explicitly.  Taking $\psi(t,U) = P_t\phi(U)$ it is not hard to show
using the It\={o} lemma that $\psi$ satisfies
\begin{align}
  \partial_t \psi  = L \psi; \quad \psi(0, U) = \phi(U)
  \label{eq:Kols:eqn:1}
\end{align}
where the generator $L$ is given by\footnote{Recall that we view  $\sigma = (\sigma_1, \ldots, \sigma_d)$ as a matrix in $\RR^{N \times d}$.}
\begin{align}
  L \phi = - \langle \nu A U + B(U,U), \nabla \psi \rangle + \frac{1}{2} \mbox{tr} (\sigma \sigma^* \nabla^2 \psi).
     \label{eq:inf:generator}
\end{align}
This system is called the \emph{Kolmogorov (forward) equation}.

The dual semigroup of $\{ P_t \}_{t \geq 0}$ acts on Borel probability measures $\mu \in Pr(\RR^N)$ according to
\begin{align*}
  P_t^* \mu(A) = \int_{\RR^N}  P_t(U_0, A) d\mu(U_0), \quad A \in \mathcal{B}(\RR^N)
\end{align*}
that is $\mu_t = P_t^*\mu$ is the probability law of solution starting from an
initial condition distributed according $\mu$.  When $\mu_t(A) = \int_A p_t(x) dx$ where $p_t \in C(\RR^N)$\footnote{In other words we are 
referring to case when $\mu$ is absolutely continuously with respect to Lebesgue measure; in more elementary 
terms this is case where solutions of \eqref{eq:fin:dim:model} are distributed as continuous random variables.}
then $p_t$ satisfies the Fokker-Plank equation
\begin{align}
  \partial_t p_t = L^* p_t; \quad p_0 \sim^d \mu.
  \label{eq:Kols:eqn:2}
\end{align}
An \emph{invariant measure} for the Markov semigroup is an element $\mu$
such that $P_t^* \mu = \mu$ for all $t \geq0$.  Such elements are (possibly generalized)
solutions of 
\begin{align}
  L^* p^* = 0
  \label{eq:inv:measure:prob:pde}
\end{align} 
and represent statistically steady states of \eqref{eq:fin:dim:model}.  Namely if
$\mu$ is an invariant measure and $U_0$ is distributed according to $\mu$
(i.e. $\mu(A) = \Prb(U_0 \in A)$, for all $A \in \mathcal{B}(\RR^N)$) then
$U(t, U_0)$ maintains this distribution $\mu$ at all later times $t > 0$.

The question of the existence of invariant measures is relatively straightforward.  It is the uniqueness
and attraction properties of these statistical states that is a subtile issue which depends heavily on 
the nonlinear structure of \eqref{eq:fin:dim:model}.  Having shown that $\mu$ is unique we infer that it is ergodic.    This means
that 
\begin{align}
  \lim_{T \to \infty} \frac{1}{T}\int_0^{T} \E \phi(U(t,U_0)) dt = \int_{\RR^N} \phi(U) d\mu(U)
  \label{eq:mu:is:ergodic}
\end{align}
which holds for every $\phi \in L^2(\RR^N, \mu)$ and $\mu$-almost every $U_0 \in \RR^N$.  One might also hope
to establish stronger attraction properties for $\mu$.    For example we say that $\mu$ is \emph{mixing} if
\begin{align}
  \lim_{T \to \infty} \E \phi(U(T,U_0)) = \int_{\RR^N} \phi(U) d\mu(U)
    \label{eq:mu:is:mixing}
\end{align}
for every $\phi \in C(\RR^N)$ and any $U_0 \in \RR^N$.  On the other hand, $\mu$ satisfies a \emph{law of large numbers}
if 
\begin{align}
  \lim_{T \to \infty} \frac{1}{T}\int_0^{T} \phi(U(t,U_0)) dt = \int_{\RR^N} \phi(U) d\mu(U)
      \label{eq:LLN}
\end{align}
almost surely for $\phi \in C(\RR^N)$ and any $U_0 \in \RR^N$.

It is worth noting that \eqref{eq:mu:is:ergodic}--\eqref{eq:LLN} are rigorous statements about 
measurement.  Under the presumption that \eqref{eq:fin:dim:model} is an accurate model
for turbulent flow (in some particular physical setting) we might suppose that $\mu$ 
uniquely characterizes the statistics of developed turbulence.   As such \eqref{eq:mu:is:ergodic}--\eqref{eq:LLN}
tell us that one may `observe' $\mu$ through physical or numerical experiments.  For example we may take
$\phi(U) = |U|^2$ or $\phi(U) = \langle U, e_k \rangle$.\footnote{In infinite dimensions these questions of measurement
are somewhat more complicated as we must address the regularity properties of observables more carefully.  Nevertheless,
similar statements can be proven.}

In view of \eqref{eq:Kols:eqn:1}--\eqref{eq:Kols:eqn:2} it is possible approach the study of statistical
properties for \eqref{eq:fin:dim:model} with the methods of (deterministic) partial differential equations.   See e.g.  \cite{Risken1989, Stroock2012}.
In particular since we are mainly interested in understanding \eqref{eq:fin:dim:model} the degenerate case where the number
of directly excited directions $d$ is much smaller than the dimension of the phase space ($d << N$) 
one would expect that the analysis of \eqref{eq:Kols:eqn:1}, \eqref{eq:Kols:eqn:2} to involve the hypoelliptic
theory of H\"ormander \cite{Hormander1967} making use of the methods of pseudodifferential calculus.  

We will take a more probabilistic approach in our analysis of invariant measures for $\{P_t\}_{t  \geq0}$ here.  This probabilistic view  
is in particular motivated by the fact that the tools that we develop been have been successfully adapted to an infinite 
dimensional setting appropriate for stochastic partial differential equations.

\section{Moment Estimates and the Existence of Stationary States}

Some basic properties of solutions of \eqref{eq:fin:dim:model} may be derived 
under \eqref{eq:pos:def:cond}--\eqref{eq:can:cond} without any further assumptions.
In particular the existence of statistically invariant states follow from basic energy balance considerations.

By applying the It\={o} formula to \eqref{eq:fin:dim:model} we have that
\begin{align*}
  d|U| + 2\langle \nu AU + B(U,U), U \rangle = |\sigma|^2  + 2\langle \sigma, U \rangle dW
\end{align*}
With \eqref{eq:pos:def:cond}, \eqref{eq:can:cond} we thus infer
\begin{align}
   |U(t)|^2 +   2 \alpha \int_0^T|U|^2 dt \leq& 
   |U_0|^2 +  |\sigma|^2T  + 2\sum_{k =1}^d \int_0^T\langle \sigma_k, U \rangle dW^k.
   \label{eq:energy:est:basic}
\end{align}
Consider the time averaged measures
\begin{align*}
   \mu_T(A) =  \frac{1}{T} \int_0^T \Prb( U(t) \in A) dt.
\end{align*}
Using the Markov inequality and \eqref{eq:energy:est:basic} we estimate
\begin{align*}
  \mu_T(B(0,R)) =& 1-  \frac{1}{T} \int_0^T \Prb( |U(t)|^2 > R^2) dt
     \geq 1 - \frac{1}{T R^2} \int_0^T \E |U(t)|^2 dt\\
     \geq& 1 - \frac{|\sigma|^2 + |U_0|^2 T^{-1}}{2 \alpha R^2}.
\end{align*}
This shows that $\{\mu_T\}_{T \geq 0}$ is tight and hence a weakly
compact sequence.  Any sub-sequential limit of this sequence 
is an invariant measure for \eqref{eq:fin:dim:model}.  To see this
observe that for any $\mu = \lim_{T_j \to \infty} \mu_{T_j}$
\begin{align*}
 \langle P_t^* \mu, \phi \rangle 
  &= \lim_{T_j \to \infty} \langle  \mu_{T_j}, P_{t}\phi \rangle 
  = \lim_{T_j \to \infty} \frac{1}{T_j} \int_0^{T_j} P_{t + s} \phi(U_0) ds
  =\lim_{T_j \to \infty} \frac{1}{T_j} \int_t^{T_j+ t} P_{t} \phi(U_0) ds\\
  &=\lim_{T_j \to \infty}  \left(\langle  \mu_{T_j}, \phi \rangle  + \frac{1}{T_j} \int_{T_j}^{T_j+ t} P_{s} \phi(U_0) ds
  -\frac{1}{T_j} \int_{0}^{t} P_{s} \phi(U_0) ds\right) = \langle \mu, \phi \rangle.
\end{align*}
Returning again to the energy balance \eqref{eq:energy:est:basic} 
one may show that family of invariant measures $\mathcal{I}$ is tight
and is hence weakly compact collection.  

\subsection{Exponential Martingales}
We can use \eqref{eq:energy:est:basic} to derive additional exponential moment estimates
which will play a crucial role below.  For this we recall the following useful martingale bound
\begin{Prop}
  Suppose that $\{M_t\}_{t \geq 0}$ is a square integrable, mean zero, martingale (relative
  to some fixed stochastic basis $\mathcal{S} = ( \Omega, \mathcal{F}, \Prb, \{\mathcal{F}_t\}_{t \geq 0}, \Prb)$) 
  and let $\langle M \rangle_t$ be the associated quadratic variation process.\footnote{Recall
  that the quadratic variation is given by
  \begin{align*}
    \langle M \rangle_t := \lim_{|\Pi_N| \to 0} \sum_{k= 0}^N |M_{t^N_k} - M_{t^N_{k-1}}|^2
  \end{align*}
  where $\Pi^N$ is any sequence of partitions of $[0,t]$ whose maximum interval goes to zero as $N \to \infty$.  By the Martingale
  representation theorem $M_t = \int_0^t f dW$ for some adapted, square integrable $f$.  One may therefore
  infer that $\langle M \rangle_t = \int_0^t |f|^2 ds$;  see \cite{KaratzasShreve} for further details.
  }
  Then 
  \begin{align}
     \Prb\left( \sup_{t \geq 0} (M_t - \frac{\gamma}{2} \langle M \rangle_t) \geq K\right) \leq e^{-\gamma K},
     \label{eq:important:MG:est}
  \end{align}
  for any $\gamma, K > 0$.
\end{Prop}
\begin{proof}
By the martingale representation theorem, there exists an adapted, square integrable 
stochastic process $f$ such that $M_t = \int_0^t f dW$.  It is not hard to see from It\=o's lemma that 
\begin{align}
   N_t := \exp\left( \gamma \int_0^t f dW - \frac{\gamma^2}{2} \int_0^t |f|^2 ds \right)
   \label{eq:exp:martingale}
\end{align}
is a local martingale.  Using Doob's inequality we therefore estimate
\begin{align*}
  \Prb\left( \sup_{s\in [0,T\wedge \tau_n]}  (M_t - \frac{\gamma}{2} \langle M \rangle_t) \geq K\right)
  = \Prb\left( \sup_{s\in [0,T\wedge \tau_n]}  N_t \geq e^{K \gamma}\right)
 = \Prb\left( \sup_{s\in [0,T]}  N_{t\wedge \tau_n} \geq e^{K \gamma}\right) \leq \frac{1}{e^{K \gamma}}
\end{align*}
for any $T > 0$ where $\tau_n$ is the localizing sequence for \eqref{eq:exp:martingale}.  Taking
$T, n \to \infty$ completes the proof.
\end{proof}

We make use of \eqref{eq:important:MG:est} with \eqref{eq:energy:est:basic} as follows.  For
any $\gamma > 0$ we have that
\begin{align*}
   |U(t)|^2 +&   2 \alpha \int_0^t|U|^2 dt    - 2\gamma \sum_{k =1}^d \int_0^t\langle \sigma_k, U \rangle^2 dt
   - |U(0)|^2 - |\sigma|^2 t \\
   &\leq
       2\sum_{k =1}^d \int_0^t\langle \sigma_k, U \rangle dW^k
    - 2\gamma \sum_{k =1}^d \int_0^t\langle \sigma_k, U \rangle^2 dt.
\end{align*}
Therefore by taking $\gamma = \gamma(|\sigma|^2, \nu)$ sufficient small and noting
that $2\sum_{k =1}^d \int_0^T\langle \sigma_k, U \rangle dW^k$ is square integrable
martingale with quadratic variation $4\sum_{k =1}^d \int_0^T\langle \sigma_k, U \rangle^2 dt$ we obtain:
\begin{Lem}\label{lem:moments}
There exists $\gamma = \gamma(|\sigma|^2, \nu)$ such that
\begin{align}
   \Prb
   \left(  \sup_{t \geq 0} \left(|U(t)|^2 +   \alpha \int_0^t |U|^2 dt - |\sigma|^2 t \right) \geq K/2 \right)
   \leq \exp(-K \gamma),
   \label{eq:exp:martingale:lead:to:awesome:1}
\end{align}
for every $K \geq 2|U_0|^2$.  Moreover 
\begin{align}
   \E \exp\left( \eta  \sup_{s \in [0,t]} \left( |U(s)|^2 + \alpha\int_0^t |U(s)|^2 ds\right) \right) 
   \leq \exp( \eta ( |U_0|^2 + |\sigma|^2 t))
   \label{eq:exp:martingale:lead:to:awesome:2}
\end{align}
where $\eta = \eta(|\sigma|^2, \nu)$ is independent of $t > 0$.
\end{Lem}

\section{Criteria for the Uniqueness of Invariant Measures}

We have seen that the existence of an invariant measure is easily established from
energy balances and a time averaging procedure.   The analysis of uniqueness and 
convergence of these measures is a much more subtle issue as exemplified by the associated
(degenerate) elliptic problem \eqref{eq:important:MG:est}. 

From a probabilistic point of view, a starting point for proving the uniqueness of invariant 
measures and for establishing attraction properties is the   Doob-Khasminskii theorem
\cite{Doob1948}, \cite{Khasminskii1960}.\footnote{One of the breakthroughs in \cite{HairerMattingly06} is that the main conclusions
of these results still hold under much weaker conditions which are more suitable for infinite dimensional systems.  
See Remark~\ref{rmk:ASF} below.}  
Before turning to a precise statement of this result
we first recall some generalities and establish notations.  Let $(H, \rho)$ be a metric space and 
take $M(H)$, respectively $Pr(H)$, to be the set of signed, respectively probability, measures
on $(H, \mathcal{B}(H))$.   We say that a sequence $\{\mu_n\}_{n \geq 0} \subset M(H)$
converges weakly if
\begin{align}
   \int_H \phi(U) d\mu_n(U) \to    \int_H d\phi(U) \mu(U) 
   \label{eq:weak:conv}
\end{align}
for all continuous bounded $\phi: H \to \RR$.\footnote{The notations and terminology from
probability theory can cause significant confusion from the point of view of functional analysis.
We are actually identifying $M(H)$ as a subspace of $C_b(H)^*$ and thus it more accurate in the language
of functional analysis to say that \eqref{eq:weak:conv} is a weak* convergence.  Some authors therefore refer 
to \eqref{eq:weak:conv} as vague convergence to avoid this confusion.}  Recall that according to 
Prokhorov's theorem a collection of elements $\mathcal{I} \subset Pr(H)$ is weakly
compact iff for every $\epsilon > 0$ there is a compact set $K_\epsilon$ such that $\mu(K_\epsilon) \geq 1 - \epsilon$
for every $\mu \in \mathcal{I}$.
For $\mu \in M(H)$ the \emph{total variation
norm} is defined equivalently as
\begin{align*}
  \|\mu\|_{TV} 
  = \frac{1}{2} \sup_{ \phi \in C_b(H), \|\phi\|_\infty \leq 1} \left| \int \phi(U) d \mu(U)\right|
  = \sup_{\Gamma \in \mathcal{B}(H)} | \mu(F)|.
\end{align*}
It is not hard to show that $\mu_1$, $\mu_2$ are mutually singular\footnote{Recall that two probability measures $\mu_1$ and $\mu_2$ are mutually
singular if there exists $A \in \mathcal{B}(H)$ such that $\mu_1(A) = 1$ and $\mu_2(A^c) = 1$.}  if and only if
$\|\mu_1 - \mu_2\|_{TV} = 1$.

Now consider a stochastically continuous and Feller Markov semigroup $\{P_t\}_{t \geq 0}$ on $(H, \rho)$.\footnote{One may 
consult e.g. \cite{ZabczykDaPrato1996} for 
details on the general setting of Markov semigroups of metric (Polish) spaces.  However the reader 
may be contented to interpret the results given
here entirely in the concrete setting introduced above for \eqref{eq:fin:dim:model}}   We say that an invariant measure $\mu$ 
of $\{P_t\}_{t \geq 0}$ is \emph{ergodic} if whenever $\Gamma \in \mathcal{B}(H)$ is such that $P_t \indFn{\Gamma} = \indFn{\Gamma}$ 
$\mu$-almost surely for every $t \geq 0$ then $\mu(\Gamma) = 0$ or $1$.   


One may show that
\begin{itemize}
\item[(a)] The set of invariant measures of $\{P_t\}_{t \geq 0}$, $\mathcal{I}$ is convex and closed in the topology of weak convergence.  
\item[(b)] The set of ergodic invariant
measures are the extremal points of $\mathcal{I}$.
\item[(c)] Any two distinct ergodic invariant measures for $\{P_t\}_{t \geq 0}$ must be mutually singular.
\end{itemize}
See e.g. \cite{ZabczykDaPrato1996} for further details.

We now recall a form of the Doob-Khasminskii theorem convenient for our purposes as follows:
\begin{Thm}\label{thm:basis:for:uniqueness}
Suppose that $\{P_t\}_{t\geq0}$ is a Markov semigroup on
a metric space $(H, \rho)$ and assume that the set of invariant measures $\mathcal{I}$
is compact in the topology of weak convergence.  Suppose that
\begin{itemize}
\item[(i)] $P_t$ is \emph{weakly irreducible} i.e. there is a point $U_0^*$ which is in the support of every 
invariant measure.\footnote{Recall that $U_0^*$ is
in the support of a measure $\mu$ if $\mu(O_{U_0^*}) > 0$ for every open set $O_{U_0^*}$ containing $U_0^*$.}
\item[(ii)] $P_t$ is \emph{strong feller}, meaning that for some $t \geq 0$, $P_t$ maps $\mathcal{M}_b(H)$ to $\mathcal{C}_b(H)$.
\end{itemize}
Then $\mathcal{I}$ has at most one element.
\end{Thm}
\begin{proof}
Let $\mu_1, \mu_2$ be two ergodic invariant measures.   We will estimate their distance in 
the total variation norm so that any bound showing that $\|\mu_1 - \mu_2\|_{TV} < 1$
implies that $\mu_1 = \mu_2$.   In order to make use of the assumed continuity of $\{P_t\}_{t \geq 0}$ we
recall that $P_t$ being strong Feller implies that $U \mapsto P_{2t}(U, \cdot)$ is continuous
in the total variation norm.  See \cite{Seidler2001,Hairer2007}.

For $\epsilon > 0$
define $m_\epsilon = \min\{\mu_1(B_\epsilon(U_0^*)), \mu_2(B_\epsilon(U_0^*))\}$
where $B_\epsilon(U_0^*)$ is the $\epsilon$ ball around $U_0^*$, the point common to the support
of all elements in $\mathcal{I}$.  In
view of weak irreducibility, $m_\epsilon > 0$.  One may decompose 
\begin{align*}
  \mu_j = (1- m_\epsilon) \tilde{\mu}_j^\epsilon + m_\epsilon \mu_j^\epsilon
\end{align*}
where $\tilde{\mu}_j^\epsilon, \mu_j^\epsilon \in Pr(H)$ and $\mu_j^\epsilon(B_\epsilon(U_0^*)) = 1$
for $j = 1,2$.\footnote{Indeed take $\mu_j^{\epsilon} = \mu_j(B_\epsilon(U_0^*) \cap \cdot)/\mu_j(B_\epsilon(U_0^*))$, 
$\tilde{\mu}_j^{\epsilon} = \mu_j(B_\epsilon(U_0^*)^c \cap \cdot)/\mu_j(B_\epsilon(U_0^*)^c)$ for $j = 1,2$. 
We then write each $\mu_j$ as a convex combination of $\mu_j^\epsilon$ and $\tilde{\mu}_j^\epsilon$.  For the measure which does not give the
minimizer of $m_\epsilon$, we adjust $\tilde{\mu}_j^{\epsilon}$ by moving inward along the line connecting $\mu_j^\epsilon$,
$\tilde{\mu}_j^{\epsilon}$.}
Since $\mu_1, \mu_2$ are invariant we estimate
\begin{align}
  \|\mu_1 - \mu_2 \|_{TV} =& \| P_{2t}^* \mu_1 - P_{2t}^* \mu_2\|_{TV}
  \leq m_\epsilon \| P_{2t}^* \mu_1^\epsilon - P_{2t}^* \mu_2^\epsilon\|_{TV} + (1 - m_\epsilon) 
  \notag\\
  =& m_\epsilon \sup_{\Gamma \in B(H)} 
  \left| \int_{B_\epsilon(U_0^*)}  P_t(U, \Gamma) d \mu_1^\epsilon(U) - \int_{B_\epsilon(U_0^*)}P_t(\bar{U}, \Gamma) d \mu_1^\epsilon(\bar{U})  \right|
  + (1 - m_\epsilon) \notag\\
  =& m_\epsilon \sup_{\Gamma \in B(H)} 
  \left| \int_{B_\epsilon(U_0^*)}   \int_{B_\epsilon(U_0^*)} (P_{2t}(U, \Gamma) - P_{2t}(\bar{U}, \Gamma)) d \mu_1^\epsilon(\bar{U}) d \mu_1^\epsilon(U)  \right|
  + (1 - m_\epsilon) \notag\\
  \leq& 1 - m_\epsilon \left(1 - \sup_{U, \bar{U} \in B_\epsilon(U_0^*)} \|P_{2t} (U, \cdot) - P_{2t} (\bar{U}, \cdot) \|_{TV} \right) \notag
\end{align}
Since $P_{2t}$ is continuous in the total variation norm, by taking $\epsilon >0$ sufficiently small 
the we obtain that
\begin{align*}
   \sup_{U, \bar{U} \in B_\epsilon(U_0^*)} \|P_t(U, \cdot) - P_t(\bar{U}, \cdot) \| \leq \frac{1}{2}
\end{align*}
and infer that $\mu_1$ and $\mu_2$ cannot be mutually singular, completing the proof.
\end{proof}

\begin{Rmk}[The asymptotic strong Feller property]
\label{rmk:ASF}
 It turns out that the finite time smoothing required by the strong Feller condition is
 too stringent for a degenerate stochastic forcing in infinite dimensions.  In 
 \cite{HairerMattingly06, HairerMattingly2008, HairerMattingly2011}
 it was shown that a \emph{time asymptotic smoothing} is all that is required for the uniqueness
 (and certain desirable attraction properties) of invariant measures.  This replacement for the first condition (i) is referred to
 as the asymptotic strong Feller (ASF) property.  We refrain from recalling precise definitions here as it would
 require introducing the Kantorovich distance and other formalities.  In application this condition allows 
 for less restrictive estimates on the gradient of the Markov semigroup, see Remark~\ref{rmk:ASF:grad:est}
 below.
\end{Rmk}
\begin{Rmk}[Mixing, Laws of Large Numbers and other convergence properties]
\label{rmk:other:attractions}
The assumptions (i), (ii) of Theorem~\eqref{thm:basis:for:uniqueness} are sufficient
(or nearly sufficient) to prove other important attraction properties al la \eqref{eq:mu:is:mixing} \eqref{eq:LLN} for the 
invariant measure.  While we will not go into further details here   
we mention the classic work \cite{MeynTweedie2009} for general results for finite dimensional systems and to
\cite{HairerMattingly2008, KuksinShirikian12} where tools relevant to proving such results in 
the infinite dimensional setting are developed.
\end{Rmk}

\section{Irreducibility}

We have set up \eqref{eq:fin:dim:model} so that the weak irreducibility property required for
Theorem~\ref{thm:basis:for:uniqueness} is relatively straightforward to obtain.\footnote{For example the situation
would be much more complicated if we added a deterministic (and time independent) forcing to \eqref{eq:fin:dim:model} in directions
not perturbed by stochastic forcing.}   We will show
that $0$ is in the support of every invariant measure of \eqref{eq:fin:dim:model}.   The proof relies
on the fact that the dynamics of unforced version of \eqref{eq:fin:dim:model} collapse to the trivial state as $t \to \infty$.  

Note that, since we would typically consider $\sigma(x) W(t) = \sum_{k =1}^d\sigma_k(x) W_k(t)$ for which the $\sigma_k$'s are 
smooth in $x$, the  proof below is easily adapted to the full 2D Navier-stokes equations (and certain other infinite 
dimensional systems).   For much more refined 
results concerning the reachability of other portions
of the phase space in both finite and infinite dimensions see 
\cite{AgrachevSarychev2004, AgrachevSarychev2005, MattinglyPardoux1, AgrachevSarychev2006,AgrachevKuksinSarychevShirikyan2007}
and containing references.

We begin by proving the following lemma
\begin{Lem}\label{lem:reachable:zero:state}
For every $R> 0$ and every $\epsilon > 0$ there exists
$T_* = T_*(R,\epsilon, \alpha)$ such that 
\begin{align}
  \inf_{U_0 \in B_R} P_t(U_0, B_\epsilon) > 0,
  \label{eq:non:trivial:disp:set}
\end{align}
for every $t > T_*$ and where $B_R = \{ U: | U | \leq R\}$.
\end{Lem}
\begin{proof}
 For any $\gamma > 0$, $t >0$ define
 \begin{align*}
 	S_{\gamma, t} = \left\{ \sup_{s \in [0,t]} | \sigma W(s)| < \gamma \right\}
 \end{align*}
 By the elementary properties of Brownian motion (see e.g. \cite{KaratzasShreve}) we have that
 $\Prb(S_{\gamma,t}) > 0$ for any choice of $\gamma, t > 0$.   For $R, \epsilon >0$ and all $t > T_*$
 (for $T_*$ to be determined shortly) we will now show that there exists 
 $\gamma > 0$ such that $ \inf_{U_0 \in B_R} P_t(U_0, B_\epsilon) \geq \Prb(S_{\gamma,t})$.

 Let $\bar{U} = U - \sigma W$.  Then $\bar{U}$ satisfies
 \begin{align}
   \frac{d\bar{U}}{dt} + \nu A\bar{U} + B(U, \bar{U}) 
   =- B(\sigma W, \sigma W) - B(\bar{U}, \sigma W) - \nu A\sigma W; 
   \quad \bar{U}(0) = U_0.
   \label{eq:shifted:equation:bm}
 \end{align}
 Using \eqref{eq:pos:def:cond}, \eqref{eq:can:cond} we obtain the estimate
 \begin{align*}
    \frac{1}{2}\frac{d|\bar{U}|^2}{dt} + \alpha| \bar{U}|^2 \leq C_1\left((|\sigma W|^2 + |\sigma W|)| \bar{U}| + |\sigma W| | \bar{U}|^2 \right)
 \end{align*}
 for a constant $C_1$ depending on $A$, $B$ and $\nu$ but independent of $t \geq 0$.  Rearranging we obtain that
 \begin{align*}
    \frac{d|\bar{U}|^2}{dt} + (\alpha - |\sigma W|) | \bar{U}|^2 \leq C_2\left((|\sigma W|^2 + |\sigma W|\right)^2.
 \end{align*}
 We now specify that $\gamma \leq \alpha/2$.  For such $\gamma$ on $S_{\gamma, t}$ we have with the Gr\"onwall lemma that
 \begin{align*}
    |U(t)|^2 \leq e^{-t (\alpha/2) }|U_0|^2 + \frac{2C_2(\gamma^2 + \gamma)^2}{\alpha}.
 \end{align*}
 Recalling that $P_t(U_0, B_\epsilon) = \Prb( |U(t,U_0)| \leq \epsilon)$ we now obtain 
 \eqref{eq:non:trivial:disp:set} by further shrinking $\gamma$ and taking $t$ sufficiently
 large as a function of $R$.  The proof is thus complete.
\end{proof} 

With Lemma~\ref{lem:reachable:zero:state} in hand we now establish that \eqref{eq:fin:dim:model}
is weakly irreducible as follows.  Fix any invariant measure $\mu$ of $\{P_t\}_{t \geq 0}$ and any 
$\epsilon > 0$.  Pick $R > 0$ such that $\mu(B_R) \geq 1/2$.    Invoking Lemma~\ref{lem:reachable:zero:state}
we choose $T_* = T_*(R, \epsilon/2)$ so that $\inf_{U \in B_R} P_{T_*}(U, B_{\epsilon/2}) > 0$. 
Now take $\phi_\epsilon: \RR^N \to [0,1]$, smooth such that
\begin{align*}
  \phi_\epsilon(U) = 
  \begin{cases}
	1 & \textrm{ when } U \in B_{\epsilon/2}, \\
	0 & \textrm{ when } U \in B_{\epsilon}^c.
  \end{cases}
\end{align*}
Observe that, by the invariance of $\mu$
\begin{align*}
  \mu(B_\epsilon) 
  	\geq& \int \phi_\epsilon(U) d \mu(U)
	= \int P_{T_*} \phi_\epsilon(U) d \mu(U)
	\geq \int_{B_R} P_{T_*} \phi_\epsilon(U) d \mu(U)\\
	\geq& \int_{B_R} P_{T_*}(U, B_{\epsilon/2}) d \mu(U)
	\geq \inf_{U \in B_R}P_{T_*}(U, B_{\epsilon/2}) \mu(B_R).
\end{align*}
In conclusion this proves:  
\begin{Prop}
  Consider \eqref{eq:fin:dim:model} under \eqref{eq:pos:def:cond}, \eqref{eq:can:cond}, \eqref{eq:noise:structure}.
  Then the trivial solution is in the support of every invariant measure of \eqref{eq:fin:dim:model}.
  In particular $\{P_t\}_{t \geq 0}$ is weakly irreducible in the sense of Theorem~\ref{thm:basis:for:uniqueness}.
\end{Prop}

\section{H\"ormander's Condition}
\label{sec:our:fav:hor}

We turn now to address the smoothing properties of $\{P_t\}_{t \geq 0}$ as required
by the second condition in Theorem~\ref{thm:basis:for:uniqueness}.  It is not
to be expected in general that equations of the form \eqref{eq:Kols:eqn:1} yield such desirable smoothing 
effects.  This is because the second order diffusion terms are not assumed to be positive definite.  
We will therefore need to develop more refined conditions concerning the interactions between the 
drift and diffusion terms in \eqref{eq:fin:dim:model}.

The study of such degenerate parabolic equations goes
back at least to Kolmogorov, \cite{Kolmogoroff1934}, and is the subject of a vast literature.  
A systematic theory of hypo-ellipticity was developed
later by H\"ormander, \cite{Hormander1967} where essentially sharp conditions concerning  
smoothing properties for systems like \eqref{eq:Kols:eqn:1}, were obtained.
Hormander's condition may be formulated for  \eqref{eq:Kols:eqn:1} as follows:  
\begin{Def}\label{def:hormander}
Define recursively
\begin{align}
  \mathcal{V}_0 &= \mbox{span} \{   \sigma_k : k = 1, \ldots, d\}\notag\\
  &\; \; \vdots \label{eq:hormander:sets} \\
  \mathcal{V}_n &= \mbox{span} \{  E, [E, F], [E, \sigma_k]: k = 1, \ldots d \textrm{ and } E \in  \mathcal{V}_{n-1} \}
  \notag
\end{align}
where $F(U) = -\nu AU - B(U,U)$ and the \emph{Lie bracket} between two (smooth) vector fields $G,H: \RR^N \to \RR^N$ is given by 
$[G, H] = \nabla H G - \nabla G H$.   
One says that \eqref{eq:Kols:eqn:1} satisfies the \emph{parabolic H\"ormander
condition} if $\cup_n \mathcal{V}_n(U) = \RR^N$ for all $U \in \RR^N$.  
\end{Def}

With this condition in mind we have the following results for our model \eqref{eq:fin:dim:model}
\begin{Prop}\label{prop:strong:feller:cond}
Suppose that \eqref{eq:fin:dim:model} satisfies \eqref{eq:pos:def:cond}, \eqref{eq:can:cond}, \eqref{eq:noise:structure}
and \eqref{eq:Hormander:cond:asmpt}; that is \eqref{eq:fin:dim:model} 
maintains the parabolic H\"ormander condition given in Definition~\ref{def:hormander}. 
Then, for every $t > 0$ and each $\phi \in C_b^1(\RR^N)$,
\begin{align}
\| \nabla P_t \phi(U_0)\| \leq C\left( \sup_{U \in \RR^N} |\phi(U)|\right),
\label{eq:grad:bound:inst:smoothing}
\end{align}
where $C = C(|U_0|,t)$ is independent of $\phi$.  As such,
the Markov semigroup $\{P_t\}_{t \geq 0}$ associated to 
\eqref{eq:fin:dim:model} satisfies the strong feller condition.
\end{Prop}

\begin{Rmk}
\label{rmk:ASF:grad:est}
In infinite dimensions it is doubtful that the instantaneous smoothing in \eqref{eq:grad:bound:inst:smoothing}
may be obtained when $\sigma$ is degenerate.  Instead following \cite{HairerMattingly06} one may establish
\begin{align}
	\| \nabla P_t \phi(U_0)\| \leq C\left( \sup_{U} |\phi(U)| + \delta(t) \sup_{U}\|\nabla \phi(U)\|\right),
	\label{eq:asym:time:grad:bound}
\end{align}
where $\delta(t) \to 0$ as $t \to \infty$.
This later bound implies the asymptotic strong Feller condition which, when combined with 
weak irreducibility, is sufficient to derive the conclusion in Theorem~\ref{thm:basis:for:uniqueness}.
\end{Rmk}

It is possible to prove Proposition~\ref{prop:strong:feller:cond} with PDE methods along the lines of \cite{Hormander1967}.  However, given the deep connection alluded to by \eqref{eq:Kols:eqn:1}, 
\eqref{eq:Kols:eqn:2} between parabolic type partial differential equations
on the one hand and white noise driven stochastic differential equations on the other, it is natural to want to have a probabilistic 
strategy to prove smoothing properties for  \eqref{eq:Kols:eqn:1}, \eqref{eq:Kols:eqn:2}. This was the initial motivation for 
the development of the Malliavin calculus initiated in the seminal work \cite{Malliavin1978}.

Below we will follow this more probabilistic approach to establish Theorem~\ref{prop:strong:feller:cond}.   
This is because these methods have been shown to be tractable for analyzing certain hypo-elliptic systems in an 
infinite dimensional setting; see \cite{HairerMattingly06, HairerMattingly2011, FoldesGlattHoltzRichardsThomann2013,FriedlanderGlattHoltzVicol2014}.

\section{Smoothing Through Control}

In this section we show how gradient bounds of the form \eqref{eq:grad:bound:inst:smoothing}
can be translated to a control problem through the usage of Malliavin calculus. 
Fix any unit vector $\xi \in \RR^N$ and observe that
\begin{align}
  \nabla P_t \phi(U_0) \xi = 
  \E \left( \nabla \phi(U(t,U_0)) \JJ_{0,t} \xi \right)
  \label{eq:grad:est:calc:1}
\end{align}
where for any $0 \leq s < t$, $\JJ_{s,t} \xi$ is the solution of the linear system
\begin{align}
  \frac{d}{dt} \rho + \nu A \rho + \nabla B(U) \rho = 0, \quad \rho(s) = \xi,
  \label{eq:J:op:def}
\end{align}
with $\nabla B(U) \rho = B(U, \rho) + B(\rho, U)$.

We would now like to find a systematic way of matching perturbations
of the initial condition in the direction $\xi \in \RR^N$ with a corresponding 
perturbation in the noise in the direction $v(\xi) \in L^2(0,t; \RR^d)$
which will allow us to `remove of the gradient' from $\phi$ in \eqref{eq:grad:est:calc:1}.   
To make sense of these vague statements we now recall some elements of the 
Malliavin calculus. See e.g. \cite{Bell1987, Nualart2009, NourdinPeccati2012} for a systematic overview
of this subject.

\subsection{Some elements of Malliavin Calculus}
The Malliavin calculus is a calculus of variations in noise paths for
random elements expressible as measurable transformations
of an underlying gaussian process.  We will now recall some basic elements of this theory: 
the derivative operator,
the divergence operator (otherwise know as the Skohorhod integral), the 
integration by parts formula and the Malliavin covariance matrix.

For $p \geq 2$ the (Malliavin) 
derivative operator  $\MD: \DD^{1,p} \subset L^p(\Omega) \to L^p(\Omega; L^2(0,t; \RR^d))$ is
defined on the class of `smooth' random variables $\mathcal{S}$ of the form
\begin{align*}
   F = F(W) = f\left( \int_0^t g_1 dW, \ldots, \int_0^t g_n dW\right) 
\end{align*}
where $f: \RR^n \to \RR$ is Schwartz class\footnote{Recall that this refers to all the $C^\infty$ functions $f$
such that $f$ and all its derivatives decay to zero faster than any polynomial at $\infty$.} and $g_1, \ldots g_n$ are deterministic
elements in $L^2(0,t; \RR^d)$ according to\footnote{We follow the notation of denoting $\MD_t F$ as
the Malliavin derivative of $F$ evaluated at time $t$.}
\begin{align}
   \MD_s F = \sum_{k =1}^n \partial_{x_k} f\left( \int_0^t g_1 dW, \ldots, \int_0^t g_n dW\right) g_k(s).
   \label{eq:action:on:smooth:rvs}
\end{align}
Starting from this definition one may show that $\MD$
is a closed operator and we denote its domain by $\DD^{1,p}$.  One immediate consequence of this
definition is that 
\begin{align}
  \textrm{ if } F \in \DD^{1,p} \textrm{ is } \mathcal{F}_r \textrm{ measurable then } \MD_s F = 0  \textrm{ for all } s > r.
  \label{eq:measurabilty:conclusion}
\end{align}
One may also show that for $F \in \DD^{1,p}$
\begin{align}
  \langle \MD F(W), v \rangle_{L^2(0,t; \RR^d)} = \lim_{\epsilon \to 0} \frac{F(W + \epsilon V) - F(W)}{\epsilon}
  \label{eq:md:bismut:def}
\end{align}
where $V(s) =  \int_0^s v(r) dr$ and $v \in L^2(0,t, \RR^d)$.\footnote{The collection of all such $V$ is referred to as
the \emph{Cameron-Martin space.}}  Observe that 
\eqref{eq:md:bismut:def} allows us to view $\MD$ as stochastic gradient operator.

A similar construction applies for random variables taking values in a Hilbert space $H$.  We define
$\MD : \DD^{1,p}(H) \subset L^p(\Omega, H) \to L^p(\Omega; L^2(0,t;\RR^d) \otimes H)$\footnote{Given Hilbert spaces $H_1$, $H_2$ we
are using the notation $H_1 \otimes H_2$ for the tensor product of $H_1$ and $H_2$.  A user friendly 
review of tensor products in the Hilbert space context can be found in \cite{Janson1997}.}
starting with its action $\MD F = \sum_j \MD F_j \otimes h_j$ on 
random element of the form $F = \sum_j F_j h_j$ with $F_j \in \mathcal{S}$, $h_j \in H$.

We will make use of two fundamental elements of the Malliavin calculus.  Firstly we have
a chain rule: for every $F \in \DD^{1,2}(\RR^N)$ and any $\phi \in C^1(\RR^N)$ we 
have that $\phi(F) \in \DD^{1,2}$ and 
\begin{align}
   \MD \phi(F) = \nabla \phi(F) \MD F.
   \label{eq:m:chain:rule}
\end{align}
We next recall an integration by parts formula which starts from an expression of duality.  
Take $\MD^*: \mbox{Dom}(\MD^*) \subset L^2(\Omega; L^2(0,t; \RR^d))
\to L^2(\Omega)$ to be the dual of $\MD$ so that 
\begin{align}
  \E \langle \MD F, v \rangle_{L^2(0,t;\RR^d)} = \E (F \MD^*v)
  \label{eq:m:IBP}
\end{align}
for every $F \in \DD^{1,2}$ and $v \in \mbox{Dom}(\MD^*)$.  
$\MD^*$ is referred to as the divergence operator
and $\MD^*v$ as the Skorokhod integral of $v$. We 
will sometimes use the notation
\begin{align}
  \MD^*v = \int_0^t v dW.
  \label{eq:sk:int}
\end{align}
It can be shown that
$\DD^{1,2}(L^2(0,t;\RR^d)) \subset \mbox{Dom}(\MD^*)$\footnote{Note that $\mbox{Dom}(\MD^*)$
may be characterized as the set of all $v \in L^2(\Omega; L^2(0,t; \RR^d))$ such that 
$|\E \langle \MD F, v \rangle| \leq C (\E |F|^2)^{1/2}$, where $C$ may depend on $v$.} and  
that \eqref{eq:sk:int} is the usual Doeblin-It\={o} integral when 
$v$ is adapted.  A second important property for our purposes
is that \eqref{eq:sk:int} satisfies a generalized
It\={o} isometry.  When $v \in \DD^{1,2}(L^2(0,t;\RR^d))$,
$\MD v \in L^2( \Omega; L^2([0,t]^2; \RR^{d\times d}))$ and 
\begin{align}
   \E \left( \int_0^t v dW\right)^2  
   =  \E \int_0^t |v|^2 dt 
   + \E \int_0^t \int_0^t \mbox{tr}( \MD_s v(r) \MD_r v(s)) dr ds.
\label{eq:gen:ito:isometry}
\end{align}
See \cite[Section 1.3.2]{Nualart2009}. Note that in view of \eqref{eq:measurabilty:conclusion} we recover
the classical It\={o} isometry when $v$ is adapted.

We now introduce $\AAA_{0,t} v$ for random elements $v \in L^2([0,t]; \RR^d)$ as the solution
of 
\begin{align}
  \partial_t \bar{\rho} +  \nu A \bar{\rho} + \nabla B(U) \bar{\rho} = \sigma v = \sum_{k=1}^d \sigma_k v_k,\quad \bar{\rho}(0) = 0
  \label{eq:mal:der:equation:SM}
\end{align}
which means that $\AAA_{0,t} v = \int_0^t \JJ_{s,t} \sigma v(s) ds$ with $\JJ_{s,t}$ defined by \eqref{eq:J:op:def}.
One may show that 
\begin{align*}
\AAA_{0,t} v = \lim_{\epsilon \to 0} \epsilon^{-1} (U(t,U_0, W + \epsilon V)  - U(t, U_0, W))
\end{align*}
where $V(t) = \int_0^t v(s) ds$ 
so that according to \eqref{eq:md:bismut:def}, $\AAA_{0,t} v = \langle \MD U(t,U_0, W), v \rangle$.\footnote{An alternative 
derivation of \eqref{eq:mal:der:equation:SM} from \eqref{eq:fin:dim:model} can be obtained by applying $\MD$ to  
\eqref{eq:fin:dim:model} in its time integrated form and making use of \eqref{eq:m:chain:rule} and \eqref{eq:action:on:smooth:rvs}.
See \cite{Debussche2011a}.}

\subsection{Deriving the Control Problem}

With these basic properties at hand we now return to the computation \eqref{eq:grad:est:calc:1}.  In view of \eqref{eq:m:chain:rule}, \eqref{eq:m:IBP}
we obtain that, for any $v \in \DD^{1,2}(L^2(0,t; \RR^d))$,
\begin{align*}
\nabla P_t \phi(U_0) \xi  &= \E \left(\nabla \phi(U(t,U_0))  \JJ_{0,t} \xi\right)
   =  \E \langle \MD \phi(U(t,U_0)), v \rangle_{L^2(0,t; \RR^d)} 
   + \E \left( \nabla \phi(U(t,U_0)) (\JJ_{0,t} \xi - \AAA_{0,t} v) \right)
   \notag\\
  &=  \E \left(\phi(U(t,U_0)) \int_0^t v dW\right)  + \E \left( \nabla \phi(U(t,U_0)) (\JJ_{0,t} \xi - \AAA_{0,t} v) \right).
\end{align*}
Therefore if for each $\xi \in \RR^N$ there exists an element $v(\xi) \in L^2([0,t], \RR^d)$ such that\footnote{In infinite dimensions
one might only hope to obtain $v(\xi)$ such that $\lim_{t \to \infty} (\JJ_{0,t} \xi - \AAA_{0,t} v(\xi)) = 0$.  With suitable bounds on $\int_0^t v dW$, this would yield asymptotic strong feller type bounds of the form \eqref{eq:asym:time:grad:bound}.  See also Remark~\ref{rmk:control:inf:d} below.}
\begin{align}
	\JJ_{0,t} \xi = \AAA_{0,t} v(\xi)
  \label{eq:control:prob}
\end{align}
we obtain that
\begin{align}
  \nabla P_t \phi(U_0) \xi = \E \left(\phi(U(t, U_0)) \int_0^t v(\xi) dW\right).
  \label{eq:after:IBP}
\end{align}
In other words we have derived the following control problem whose solution immediately
yields Proposition~\ref{prop:strong:feller:cond} as a corollary:
\begin{Prop}\label{lem:control:prob:concrete}
  For every $t > 0$ and $\xi \in \RR^N$ there exists $v(\xi) \in \DD^{1,2}( L^2(0,t, \RR^d))$ such that 
  the solution of 
  \begin{align*}
    \frac{d}{dt} \tilde{\rho} + \nu A \tilde{\rho} + \nabla B(U) \tilde{\rho} = -\sum_{k = 1}^d \sigma_k v_k(s), \quad \tilde{\rho}(0) = \xi,
  \end{align*}
  is zero at time $t$ and 
  \begin{align}
    \sup_{|\xi| = 1} \E\left|\int_0^t v(\xi) dW\right|^2 < \infty.
    \label{eq:cost:of:control:was:not:high}
  \end{align}
\end{Prop}

\subsection{The Malliavin Matrix}

A solution $v = v(\xi)$ of \eqref{eq:control:prob} is derived 
as follows.   Take the ansatz that $v = \AAA_{0,t}^* \eta$ where $\AAA_{0,t}^*$
is the adjoint of $\AAA_{0,t}$ and $\eta \in \RR^N$.
Observe that $\AAA_{0,t}^*$ almost surely maps $\RR^N$ to $L^2(0,t, \RR^d)$ and acts according to 
\begin{align*}
   (\AAA_{0,t}^* \eta)(s) = \sigma^*\JJ_{s,t}^*\eta
\end{align*}
where $\JJ_{s,t}^*$ is the dual of $\JJ_{s,t}$.  Note that $\JJ_{s,t}^*\eta$ may be seen to be the solution of 
the final value problem
\begin{align}
  - \frac{d}{dt} \rho^* + \nu A \rho^* + (\nabla B(U))^*\rho = 0, \quad \rho^*(t)= \eta.
  \label{eq:J:adj:evol}
\end{align}
With this ansatz we find that 
\begin{align*}
  \MM_{0,t} \eta :=  \AAA_{0,t}\AAA_{0,t}^* \eta = \JJ_{0,t} \xi.
\end{align*}
$\MM_{0,t}$ is called the \emph{Malliavin covariance matrix}.  Under the 
(as yet unjustified) assumption that $\MM_{0,t}$ is invertible
we obtain the following formal solution of \eqref{eq:control:prob}
\begin{align}
  v(\xi) =  \AAA_{0,t}^* \MM_{0,t}^{-1} \JJ_{0,t} \xi.
  \label{eq:control:solution}
\end{align}

\subsection{Assessing the cost of control}
\label{sec:cost:control}

In order to complete the proof of Proposition~\ref{lem:control:prob:concrete} we need
to show that $\MM_{0,t}^{-1}$ is invertible and moreover to establish 
\eqref{eq:cost:of:control:was:not:high}.  Below in Section~\ref{sec:analysis:M} we will prove\footnote{For 
didactic reasons we will only establish Proposition~\ref{prop:main:spec:bound} in full detail for an interesting special case.  See Definition~\ref{def:Hormander:special} below.}
\begin{Prop}
\label{prop:main:spec:bound}
For all $q \geq 1$ there exists $\epsilon_0(q, \nu, |\sigma|) > 0$ and $C = C(q, \nu, |\sigma|, |U_0|)$ such that
\begin{align}
  \Prb \left( \inf_{\eta \in \RR^N \setminus \{0\}}  \frac{\langle \MM_{0,t} \eta, \eta \rangle}{| \eta |^2} \geq \epsilon  \right) 
  \geq 1 - C \epsilon^q 
  \label{eq:Mall:spec:est}
\end{align}
for all $\epsilon > 0$.
\end{Prop}
\noindent Since $\MM_{0,t}$ is symmetric and positive, it is clear from this proposition that $\MM_{0,t}$ is almost surely invertible.  Moreover, we have following the  corollary:
\begin{Cor}\label{cor:mom:M:cor}
For every $p \geq 2$, there exists $K = K(p, \nu, |\sigma|, |U_0|) < \infty$ such that
\begin{align}
   \E \| \MM_{0,t}^{-1}\|^p \leq K < \infty.
   \label{eq:finite:moments:Mal:M}
\end{align}
\end{Cor}
\begin{proof}[proof of Corollary~\ref{cor:mom:M:cor}]
Observe that for any $\epsilon > 0$ 
\begin{align*}
  \left\{ \inf_{\eta \in \RR^N \setminus \{0\}}  \frac{\langle \MM_{0,t} \eta, \eta \rangle}{| \eta |^2} \geq \epsilon  \right\}
  \subset \left\{ \frac{1}{\epsilon^2} \geq \| \MM_{0,t}^{-1}\|^2 \right\}.
\end{align*}
Hence using \eqref{eq:Mall:spec:est} we obtain that for any $r \geq 1$
there exists an $x_0 > 0$, $C > 0$ depending on $r$ and the data such 
\begin{align*}
\Prb \left(  \| \MM_{0,t}^{-1}\| > x \right) \leq  C x^{-r}
\end{align*}
for all $x \geq x_0$.  The desired bound \eqref{eq:finite:moments:Mal:M} now
follows from the elementary identity $\E |X|^p = p \int_0^\infty x^{p-1} \Prb(|X| > x) dx$.
\end{proof}

With Corollary~\ref{cor:mom:M:cor} in hand we establish \eqref{eq:cost:of:control:was:not:high} using the generalized It\={o}
isometry \eqref{eq:gen:ito:isometry} as follow.  
Start with the first term in \eqref{eq:gen:ito:isometry}
\begin{align*} 
 \E\int_0^t |\AAA_{0,t}^* \MM_{0,t}^{-1} \JJ_{0,t} \xi|^2 ds
 \leq& (\E\|\MM_{0,t}^{-1}\|^6 \E\|\JJ_{0,t}\|^6)^{1/3}
 \int_0^t (\E \| \sigma^* \JJ_{s,t}^*\|^6)^{1/3} ds\notag\\
 \leq& C 
 \left(\E \|\MM_{0,t}^{-1}\|^6 
         \E \|\JJ_{0,t}\|^6 
       \sup_{s \in [0,t]} \E \| \JJ_{s,t}\|^6   \right)^{1/3}.
\end{align*} 
We therefore need a suitable bound for $\JJ_{s,t}$.  Working from \eqref{eq:J:op:def} we find
\begin{align*}
  \frac{d}{dt} | \rho|^2 + \alpha | \rho |^2 \leq  |\langle B(\rho, U), \rho \rangle|
  \leq C |\rho|^2 |U|.
\end{align*}
Rearranging and applying the Gr\"onwall lemma and 
the bound \eqref{eq:exp:martingale:lead:to:awesome:2} 
we infer that for any $p, \eta > 0$
\begin{align}
  \E \left(\sup_{s \in [0,t]} \| \JJ_{s,t}\|^p \right) \leq C\E\exp\left(\eta\int_0^t |U|^2ds\right)  \leq C \exp(\eta |U_0|^2) < \infty.
  \label{eq:J:bnd}
\end{align}
for a universal constant $C = C(p,\eta, |\sigma|, t)$.

In order to 
treat the second term in \eqref{eq:gen:ito:isometry} we use the Malliavin chain rule to obtain
\begin{align}
  \MD v(\xi) &= \MD \AAA_{0,t}^* \MM_{0,t}^{-1} \JJ_{0,t} \xi 
  - \AAA_{0,t}^* \MM_{0,t}^{-1}\MD \MM_{0,t} \MM_{0,t}^{-1} \JJ_{0,t} \xi 
  + \AAA_{0,t}^* \MM_{0,t}^{-1}\MD \JJ_{0,t} \xi \notag\\
  &= \MD \AAA_{0,t}^* \MM_{0,t}^{-1} \JJ_{0,t} \xi 
  - \AAA_{0,t}^* \MM_{0,t}^{-1}(\MD \AAA_{0,t}\AAA_{0,t}^* + \AAA_{0,t} \MD \AAA_{0,t}^* )\MM_{0,t}^{-1} \JJ_{0,t} \xi 
  + \AAA_{0,t}^* \MM_{0,t}^{-1}\MD \JJ_{0,t} \xi.
  \label{eq:mall:der:yields:MESS}
\end{align}
A computation making use of 
\eqref{eq:measurabilty:conclusion} yields that
\begin{align}
  \MD_\tau \JJ_{s,t}\xi =
  \begin{cases}
     \JJ^{(2)}_{\tau,t}( \sigma, \JJ_{s,\tau} \xi) & \textrm{ when } s < \tau,\\
     \JJ^{(2)}_{s,t}(\JJ_{\tau, s} \sigma, \xi) &  \textrm{ when } s \geq \tau,
  \end{cases}
  \label{eq:MD:JJ}
\end{align}
where for $\xi, \xi' \in \RR^N$, $\tilde{\rho} = \JJ_{s,t}^{(2)}(\xi, \xi')$  is the solution of
\begin{align*}
 \frac{d}{dt} \tilde{\rho} + \nu A \tilde{\rho} + B(u,\tilde{\rho}) + B(\tilde{\rho},u)  + 
 B(\JJ_{s,t} \xi ,\JJ_{s,t} \xi') + B(\JJ_{s,t} \xi',\JJ_{s,t} \xi) = 0, \quad \tilde{\rho}(s) = 0.
\end{align*}
Combining \eqref{eq:mall:der:yields:MESS} with \eqref{eq:MD:JJ} and using \eqref{eq:exp:martingale:lead:to:awesome:2} 
one now obtains suitable bounds for the second term in \eqref{eq:gen:ito:isometry}.  This completes
the proof of \eqref{eq:cost:of:control:was:not:high} (modulo the proof of Proposition~\ref{prop:main:spec:bound}, given below).

\begin{Rmk}[Control In Infinite Dimensions]
\label{rmk:control:inf:d}
  In infinite dimensions it is not clear that the Malliavin matrix is invertible
  and therefore we cannot employ \eqref{eq:control:solution} to obtain gradient 
  bounds, a la \eqref{eq:grad:bound:inst:smoothing}.  However when a slightly 
  more uniform version of the H\"ormander condition is satisfied a spectral 
  bound similar to Proposition~\ref{prop:main:spec:bound} still holds (and may
  be established using the methods outlined below in Section~\ref{sec:analysis:M}).
  In \cite{HairerMattingly06}, for the analysis of the 2D stochastic Navier-Stokes equations 
  the solution \eqref{eq:control:solution} is replaced with a control involving a suitable 
  regularization of $\MM_{0,t}$.  By combining this modified $v$, the spectral analysis of the Malliavin 
  matrix and a Foias-Prodi type bound, \cite{FoiasProdi1967}, to control small scales with the parabolic dissipation coming from 
  viscous terms, an estimate of the form \eqref{eq:asym:time:grad:bound} was established.  
  This basic approach has proven to be effective for the analysis of a number of other nonlinear systems; see 
  \cite{HairerMattingly2011, FoldesGlattHoltzRichardsThomann2013,FriedlanderGlattHoltzVicol2014}.  Note that in 
  each of these works different infinite dimensional variations of H\"ormander's classical condition are explored.
\end{Rmk}

\section{A Spectral Analysis of the Malliavin Matrix}
\label{sec:analysis:M}

Our final task is to establish the spectral bound on $\MM_{0,t}$ given
in Proposition~\ref{prop:main:spec:bound}.  This is where the 
`rubber meets the road' for establishing the strong feller property in the sense 
that it is at this point in the analysis that we make 
use of fact that the H\"ormander bracket condition (cf. Definition~\ref{def:hormander})
is satisfied.

Broadly speaking the proof may be understood as an iterative proof by contradiction
on sets of quantitatively large measure.  First we show that if $\MM_{0,t}$
has a small eigenvalue then (with large probability) the corresponding (unit) eigenvector 
must have a small inner product with all elements in the sets given in 
\eqref{eq:hormander:sets}.  On the other hand we have made the assumption that 
the elements of \eqref{eq:hormander:sets} form a spanning set for $\RR^N$.  We therefore
conclude that $\MM_{0,t}$ does not posses such a small eigenvalue, as desired.

\subsection{A special case of the H\"ormander Condition}
In order to focus on the main ideas we will prove 
Proposition~\ref{prop:main:spec:bound} only for a special case; we emphasize that
Proposition~\ref{prop:main:spec:bound} holds in the generality of the conditions given in 
Proposition~\ref{prop:strong:feller:cond} and that basic approach of the analysis is the same.
\begin{Asmp}
\label{def:Hormander:special}
Let:
\begin{align}
  \mathcal{W}_0 &:= 
   \{\sigma_j : j = 1, \ldots d \} \notag\\
  			&\;  \vdots 
			\label{eq:simplified:hor}\\
  \mathcal{W}_n &:= 
  \{[[F,\psi],\sigma_j] = B(\psi, \sigma_{j}) + B(\sigma_{j}, \psi)
  : j = 1, \ldots d , \psi \in \mathcal{W}_{n-1}\}	 \cup \mathcal{W}_{n -1}
  \notag
\end{align}
We suppose that there exists an $n$ such that $\mbox{span}(\mathcal{W}_n ) = \RR^N$.
\end{Asmp}

\begin{Rmk}[Domain of Applicability of Assumption~\ref{def:Hormander:special}]
The special case of the H\"ormander bracket condition given in Assumption~\ref{def:Hormander:special}
is sufficient to directly address several interesting situations namely the 2 and 3 dimensional Navier-Stokes equations
with a stochastic forcing acting through trigonometric basis elements.  This is because, in these cases $B(e_j, e_k) \sim e_{j + k}$
(where $\{e_k\}_{k}$ represents a trigonometric basis for $L^2(\mathbb{T}^2)$); see \cite{EMattingly2001, Romito2004}.  
Other situations of interest require a more involved bracket structure.  
See \cite{FoldesGlattHoltzRichardsThomann2013, FriedlanderGlattHoltzVicol2014}.
While the proof given below does not directly cover these later cases the main ideas are 
the same.  See \cite{HairerMattingly2011, FoldesGlattHoltzRichardsThomann2013}
for further details.
\end{Rmk}

\begin{Rmk}[Norris' Approach]
Note that the analysis below will follow the approach developed in \cite{MattinglyPardoux1,BakhtinMattingly2007,HairerMattingly2011}.
While it it would be more direct to use the method of Norris \cite{Norris1986} (see also \cite{Nualart2009, Hairer2011} for a more 
recent account) this would require an inversion of $\JJ_{0,t}$ defined by \eqref{eq:J:op:def}.
Such an inversion is of course a severe restriction for addressing infinite dimensional problems.  
\end{Rmk}

\subsection{Preliminary Bounds}

We next establish two preliminary estimates.  In particular these estimates provide a means of obtaining quantitive bounds
corresponding to brackets against $F$ and to brackets against $\sigma_j$ respectively.

The first estimate (taken from \cite{FoldesGlattHoltzRichardsThomann2013})
makes use of the basic fact that if a function's supremum is small then 
either its derivative is small or its $C^{1,\alpha}$ norm must be large.  
\begin{Lem}
\label{lem:small:to:small:prime:wrapper}
Fix any $\alpha >0, t > 0$. Consider a family of random functions $g_\phi : \Omega \times [0,t] \to \RR$ indexed by
a countable collection of elements $\Dtest$ such that $g_\phi \in C^{1,\alpha}([0,t])$ 
almost surely for each $\phi \in \Dtest$.  Take 
\begin{align}
  \Omega_{\phi, \epsilon} = 
  \left\{  \|g_\phi\|_{L^\infty} \leq \epsilon \textrm{ and } 
  	\|g_\phi'\|_{L^\infty} \geq \epsilon^{\alpha/(2(1 + \alpha))} \right\},
	\label{eq:exceptional:sets:derivatives}
\end{align}
for any $\epsilon > 0$ and $\phi \in \Dtest$ and let $\Omega_\epsilon = \cup_{\phi \in \Dtest} \Omega_{\phi, \epsilon}$.  Then there exists $\epsilon_0 = \epsilon_0(\alpha, t)$ such
that
\begin{align}
  \Prb(\Omega_\epsilon) 
  \leq 4^{\frac{2q(1+ \alpha)}{\alpha}} \epsilon^q 
  \E \left( \sup_{\phi \in \Dtest} \|g_\phi \|_{C^{1,\alpha}}^{2q /\alpha} \right),
   \label{eq:exceptional:sets:est:derivatives}
\end{align}
for every $\epsilon < \epsilon_0$.  In particular 
\begin{align*}
  \Omega_{\epsilon}^c =  \bigcap_{\phi \in \Dtest}
  \left\{  \|g_\phi\|_{L^\infty} > \epsilon \,\textrm{ or } \,
  	\|g_\phi'\|_{L^\infty} < \epsilon^{\alpha/(2(1 + \alpha))} \right\},
\end{align*}
so that on  $\Omega_{\epsilon}^c$ the implication
\begin{align}
  \|g_\phi\|_{L^\infty} \leq \epsilon \quad \Rightarrow \quad \|g_\phi'\|_{L^\infty} < \epsilon^{\alpha/(2(1 + \alpha))}
\label{eq:user:friendly:deriv:implic}
\end{align}
holds for any $\phi \in  \Dtest$.
\end{Lem}
\begin{proof}[Proof of Lemma~\ref{lem:small:to:small:prime:wrapper}]
  Recall the interpolation inequality
   \begin{align*}
    \| f '\|_{L^\infty} \leq 4  \|f\|_{L^\infty} \max \left\{\frac{1}{t},
    \| f\|_{L^\infty}^{-1/(1 + \alpha)}
    \| f' \|_{C^{\alpha}}^{1/(1 + \alpha)}  \right\} \,,
    \label{eq:holder:interp:0}
 \end{align*}
 which is valid for any $f \in C^{1,\alpha}([0,t])$.   This means that
 on $\Omega_\epsilon$
 \begin{align*}
  \epsilon^{\frac{\alpha}{2(1 + \alpha)}} 
    \leq 4 \max \left\{\epsilon \frac{1}{t} ,
    \epsilon^{\alpha/(1 + \alpha)}
    \| g'_{\phi} \|_{C^{\alpha}}^{1/(1 + \alpha)}  \right\}
 \end{align*}
 for some $\phi \in \Dtest$.  Hence, for every $\epsilon \leq \epsilon_0(\alpha, t) := \left(\frac{t}{4}\right)^{\frac{2(1+\alpha)}{2 + \alpha}}$,
 we have that 
 \begin{align*}
   \epsilon^{-\frac{\alpha}{2(1 + \alpha)}} 
    \leq 4 \| g'_{\phi} \|_{C^{\alpha}}^{1/(1 + \alpha)} 
 \end{align*}
on $\Omega_\epsilon$ for some $\phi \in \Dtest$.  We thus infer,
 \begin{align*}
  \Prb(\Omega_\epsilon) \leq \Prb\left( 4^{\frac{2(1+ \alpha)}{\alpha}} \sup_{\phi \in \Dtest}  \|g_\phi \|_{C^{1,\alpha}}^{2/\alpha} \geq \epsilon^{-1}  \right)
 \end{align*}
 for every $\epsilon \leq  \epsilon_0$.
 The desired bound \eqref{eq:exceptional:sets:est:derivatives} now follows from Chebshev's inequality, completing the proof.
\end{proof} 

The next bound on first order Wiener polynomials is a special case of a more general result found
in \cite{HairerMattingly2011}.\footnote{Lemma~\ref{lem:wiener:poly:simple:case}  
can be extended to polynomials of arbitrary order and the bounds presented can be significantly sharpened.  We
restrict our attention to this special case (which is sufficient for our purposes here) for clarity of presentation.}  It expresses
the intuitive fact that a Wiener polynomial can be small only if either each of its coefficients are small or if alternatively each coefficient 
fluctuates wildly.

\begin{Lem}
\label{lem:wiener:poly:simple:case}
Consider
\begin{align*}
  F(t) = A_0(t) + \sum_{j=1}^d A_j(t) W^j(t)
\end{align*}
where each $A_j : \Omega \times [0,t] \to \RR$ is an arbitrary continuous stochastic process
and $W = (W^{1}, \ldots, W^{d})$ is a standard ($d$-dimensional) brownian motion.
Fix any $q \geq 1$ and any $r > 0$.  Then, for each $\epsilon > 0$, there exist a measurable set $\Omega_\epsilon$ 
such that
\begin{align}
  \Prb(\Omega_{\epsilon}) \leq C\epsilon^q,
\end{align}
where $C = C(q, r,t)$ is independent of $\epsilon$ and the coefficients $A_0, \ldots, A_d$, and so that on $\Omega_\epsilon^c$
\begin{align}
   \sup_{s \in [0,t]} |F(t)| \leq \epsilon^r
   \quad \Rightarrow \quad
   \begin{cases}
   \textrm{either}& \max\limits_{j =1, \ldots, d} \sup\limits_{s \in [0,t]} |A_j(t)| \leq (d + 3)\epsilon^{r/16},\\
   \textrm{or}&   \max\limits_{j =0, \ldots, d} \sup\limits_{\substack{s_1, s_2\in [0,t]\\ s_1 \not= s_2}}
   				\frac{|A_j(s_1)- A_j(s_2)|}{|s_2-s_1|} \geq \epsilon^{-r/16}.
   \end{cases}
   \label{eq:good:wiener:behavior}
\end{align}
\end{Lem}
\begin{proof}
 Fix any $\epsilon >0$.  We will proceed by proving that there exists a set $\Omega_\epsilon$
 with $\Prb(\Omega_\epsilon) = C \epsilon^q$ such that off this set if 
 \begin{align}
  \|F\|_{L^\infty} < \epsilon^r \quad \textrm{ and } \quad \sup_{j= 0, \ldots d} \| A_j\|_{Lip} \leq \epsilon^{-r/16}
  \label{eq:good:F:A:asp}
\end{align}
then
\begin{align}
  \sup_{j= 1, \ldots, d} \| A_j\|_{L^\infty} \leq (d + 3) \epsilon^{r/16}.
    \label{eq:good:F:A:imp}
\end{align}
 
 Split the interval $[0,t]$ into subintervals of length $\epsilon^{\kappa}$ where 
 $\kappa > 0$ will be determined presently. Observe that
 \begin{align*}
   F(s) = F(m\epsilon^{\kappa}) + A_0(s) - A_0(m \epsilon^{\kappa}) 
   +\sum_{j = 1}^d  A_j(m \epsilon^{\kappa}) (W^j(s) - W^j(m \epsilon^{\kappa}))
             + \sum_{j = 1}^d (A_j(s) - A_j(m \epsilon^{\kappa}))W^j(s),
 \end{align*}
 for any $s \in [0,t]$, $\epsilon > 0$ and every integer $m \in [0, t \epsilon^{-\kappa}]$.  Now define
 \begin{align*}
    \Omega_{\epsilon, 1} := \left\{ \|W\|_{L^\infty} \geq \epsilon^{-\kappa/4} \right\}.
 \end{align*}
 Using Doob's inequality we infer that $\Prb(\Omega_{\epsilon,1}) \leq C\epsilon^q$ 
 for a constant $C = C(t, \kappa, q)$ independent of $\epsilon$.
 Supposing that \eqref{eq:good:F:A:asp} holds we infer that, off of $\Omega_{\epsilon, 1}$,
 \begin{align}
   \left|\sum_{j = 1}^d  A_j(m \epsilon^\kappa) (W_j(s) - W_j(m \epsilon^\kappa)) \right| &\leq 
   2 \|F\|_{L^{\infty}} +   | A_0(s) - A_0(m \epsilon^\kappa)| + 
   \sum_{j = 1}^d | A_j(s) - A_j(m \epsilon^\kappa)| |W_j(s)| \notag\\
   &\leq 2 \epsilon^r + (d+1)\epsilon^{3\kappa/4 -r/16} \notag
 \end{align}
 for all $s \in [m \epsilon^\kappa, (m+1) \epsilon^\kappa]$ and each $m = 0, \ldots, t \epsilon^{-\kappa}$.  
 Using the Brownian scaling we may rewrite this as
 \begin{align}
    |\tilde{A}^m||\langle u_{\tilde{A}^m}, \tilde{W}^m(\tau) \rangle| \leq 2 \epsilon^{r- \kappa/2} + (d+1)\epsilon^{\kappa/4 -r/16}
    \label{eq:lower:bnd:A}
 \end{align}
 valid for $\tau \in [0,1]$ and $m = 0, \ldots, t \epsilon^{-\kappa}$ where
 \begin{align*}
   \tilde{A}(t) := (A_1(t), \ldots, A_d(t)), \quad \tilde{A}^m := \tilde{A}(m \epsilon^\kappa), \quad u_{\tilde{A}^m} := \frac{\tilde{A}^m}{|\tilde{A}^m|}
 \end{align*}
 and
 \begin{align*}
   \tilde{W}^m(\tau) =  \frac{W((m + \tau)\epsilon^\kappa) - W(m \epsilon^\kappa)}{\epsilon^{\kappa/2}}.
 \end{align*}

In order to take advantage of \eqref{eq:lower:bnd:A} to infer \eqref{eq:good:F:A:imp} from \eqref{eq:good:F:A:asp} we 
make use of the \emph{small ball probability estimate} 
 \begin{align}
   \Prb\left( \sup_{t \in [0,1]} |B(t)| \leq \epsilon \right) \leq \frac{4}{\pi}\exp\left( -\frac{\pi^2}{8 \epsilon^2}\right)
   \label{eq:small:ballz}
 \end{align}
which holds for every $\epsilon > 0$ and is valid for any 1-d standard brownian motion $B$; see \cite{CiesielskiTaylor1962} and 
also e.g. \cite{Li2001,BerglundGentz2006}
 for further references.  Now consider any $d$-dimensional brownian motion $U$.    For a given $\epsilon > 0$ pick
 pick points $u_1, \ldots, u_{M_\epsilon}$ on the unit sphere $\mathbb{S}^d$ so that any $u$ on $\mathbb{S}^d$
 is a distance at most $\epsilon^2$  from one of these points.    Define
 \begin{align*}
   \Omega_{\epsilon, U}
   := \bigcup_{j=1}^{M_\epsilon} \biggl\{ \sup_{t \in [0,1]} |\langle U, u_j\rangle | \leq 2\epsilon \biggr\}
   \cup \{ |U| \geq \epsilon^{-1}\}.
 \end{align*}
 Notice that the points $u_j$ defining these sets can be chosen in such a fashion that
 $M_\epsilon$, the number of points, grows at most algebraically in $\epsilon^{-1}$ as $\epsilon \to 0$.  Using this 
 observation and \eqref{eq:small:ballz} we find that, for any $p \geq 1$
 \begin{align*}
    \Prb\left( \Omega_{\epsilon, U} \right) < C \epsilon^{p}
 \end{align*}
 for some suitable  $C = C(p) > 0$ which can be chosen independently of $\epsilon$.  Notice furthermore that, on $\Omega_{U, \epsilon}^c$, 
 given any $u \in \mathbb{S}^d$ we have that $|\langle U, u\rangle|  \geq |\langle U, u_j \rangle | - |U| |u - u_j|$
 for any one of our selected points $u_j$.
 Hence, for any $u \in \mathbb{S}^d$, $|\langle U, u\rangle| \geq \epsilon$ on $\Omega_{\epsilon, U}$.

 Taking now
 \begin{align*}
    \Omega_{\epsilon, 2} = 
      \bigcup_{m=1}^{t\epsilon^{-\kappa}}
          \Omega_{\epsilon^{\kappa/8}, W^m}
 \end{align*}
 we infer from the preceding discussion that, for the desired value of $q > 0$ there exists $C = C(q) > 0$ s.t.
 \begin{align*}
  \Prb( \Omega_{\epsilon, 2}) \leq 
   C \epsilon^q
 \end{align*}
 for every $\epsilon > 0$.  Thus, off of $\Omega_\epsilon = \Omega_{\epsilon,1} \cup \Omega_{\epsilon,2}$, 
 whenever \eqref{eq:good:F:A:asp} holds we have
 \begin{align*}
       |\tilde{A}^m|  \leq 2 \epsilon^{r- 5\kappa/8} + (d+1)\epsilon^{\kappa/8 -r/16}
 \end{align*}
 for each $m  = 0, \ldots, t\epsilon^{-\kappa}$.  As such, if \eqref{eq:good:F:A:asp} holds, on $\Omega_\epsilon^{c}$
 we have that for any $s \in [m \epsilon^{\kappa}, (m+1) \epsilon^\kappa]$
 \begin{align*}
   |\tilde{A}(s)| \leq | \tilde{A}(s)- \tilde{A}^m| + |\tilde{A}^m| 
   \leq \epsilon^{\kappa - r/16} + 2 \epsilon^{r- 5\kappa/8} + (d+1)\epsilon^{\kappa/8 -r/16}.
 \end{align*}
 Since this bound is maintained for each $m  = 0, \ldots, t\epsilon^{-\kappa}$
 we infer
 \begin{align*}
   \max_{j =1, \ldots, d} \|A_j\|_{L^\infty} \leq (d + 3) \epsilon^{\kappa/16}
 \end{align*}
 on $\Omega_\epsilon^{c}$ by choosing $\kappa = r$.  The proof is now complete.

\end{proof}

\subsection{Translating Bracket Operations to Quantitative Bounds}

We now make use of the machinery developed in the previous section to
prove two lemmata which give quantitative bounds for individual
bracket operations. In the following section we will string together these
individual steps to complete the proof of Proposition~\ref{prop:main:spec:bound}.

We begin by showing that if the Malliavin matrix has a small eigenvalues then all of the directly forced directions
must also be small (with high probability).  This corresponds to building the initial set
$\mathcal{W}_0$ in the sequence of sets defined in \eqref{eq:simplified:hor}.

\begin{Lem}\label{lem:inital:step}
For all $q \geq 1$ and any countable collection $\Dtest \subset \SN$, there exists measurable sets $\Omega_{\epsilon}$ defined
for every $\epsilon >0$ with
\begin{align}
  \Prb(\Omega_{\epsilon}) \leq C \epsilon^q 
  \label{eq:bad:set:small:initial}
\end{align}
where $C = C(\nu, |\sigma|, |U_0|)$ and such that on $\Omega_{\epsilon}^c$
\begin{align}
 \langle \MM_{0,t} \eta, \eta \rangle \leq \epsilon 
  \quad \Rightarrow  \quad
  \max_{j = 1,\ldots, d} \sup_{s\in [0, t]} |\langle \JJ_{s,t}^* \eta, \sigma_j\rangle| \leq \epsilon^{1/4}
  \label{eq:implication:initial:step}
\end{align}
for every $\eta \in \Dtest$.
\end{Lem}
\begin{proof}
  We begin by observing that, for any $\eta \in \RR^N$,
  \begin{align*}
     \langle \MM_{0,t} \eta, \eta \rangle = \int_0^t |\sigma^* \JJ_{s,t}^* \eta|^2 ds
     = \sum_{j =1}^d \int_0^t \langle \JJ_{s,t}^* \eta, \sigma_j \rangle^2 ds.
  \end{align*}
  For each $\eta \in \Dtest$ we take 
  \begin{align*}
  g_\eta(s) = \sum_{j =1}^d \int_0^s \langle \JJ_{r,t}^* \eta, \sigma_j \rangle^2 dr
  \end{align*}
  and observe
  \begin{align*}
  g_\eta'(s) = \sum_{j =1}^d \langle \JJ_{s,t}^* \eta, \sigma_j \rangle^2,\quad 
  g_\eta''(s) = 2\sum_{j =1}^d \langle \JJ_{s,t}^* \eta, \sigma_j \rangle
  \langle \frac{d}{ds}\JJ_{s,t}^* \eta, \sigma_j \rangle =2\sum_{j =1}^d \langle \JJ_{s,t}^* \eta, \sigma_j \rangle
  \langle \JJ_{s,t}^* \eta, [F,\sigma_j] \rangle;
  \end{align*}
  see \eqref{eq:J:adj:evol} for the final equality.  In view of applying Lemma~\ref{lem:small:to:small:prime:wrapper}
  we now estimate
  \begin{align}
    \E \left( \sup_{\eta \in \Dtest} |g_\eta|_{C^{2}}^{2} \right)
    \leq 2| \sigma| \E \sup_{s \in [0, t]}\left( \| \JJ_{s,t}\|^2  \sum_{j = 1}^d| \nu A \sigma_j + B(\sigma_j, U) + B(U, \sigma_j))| \right)
    \leq C \E\exp(\expCon |U_0|^2),
    \label{eq:U:quant:est:1}
  \end{align}
  where the final inequality is obtained from \eqref{eq:exp:martingale:lead:to:awesome:2}, \eqref{eq:J:bnd} and  the constant 
  $C = C(|\sigma|, \kappa, t)$ is independent of $\eta \in \Dtest \subset \SN$.  By applying Lemma~\ref{lem:small:to:small:prime:wrapper}
  we infer the existence of sets $\Omega_\epsilon$, defined for all $\epsilon > 0$, which satisfy \eqref{eq:bad:set:small:initial}.   
  Moreover, off of each $\Omega_\epsilon$, the implication \eqref{eq:implication:initial:step} holds for every $\eta \in \Dtest$, 
  cf. \eqref{eq:user:friendly:deriv:implic}.  The proof is now complete.
\end{proof}

We next show that if a set of direction $\dir^{(l)} \in \RR^N$, $l = 1, \ldots, m$ is small then, with 
high probability, $[[F, \dir^{(l)}], \sigma_j]$ is small for every $j = 1, \ldots, d$ and every
$l = 1, \ldots, m$.    We will use this lemma below to build the new elements in $\mathcal{W}_n$ from $\mathcal{W}_{n-1}$
defined in \eqref{eq:simplified:hor}.
\begin{Lem}
\label{lem:direction:to:F:bomb}
Fix any collection $\dir^{(l)} \in \RR^N$, $l = 1, \ldots, m$.  For every $\delta > 0$, $q \geq 1$ 
and any countable collection $\Dtest \subset \SN$,
there exists measurable sets $\Omega_{\epsilon}$ defined for every 
 $\epsilon >0$ such that
\begin{align}
  \Prb(\Omega_{\epsilon}) \leq C \epsilon^q 
  \label{eq:good:set:cond}
\end{align}
and so that on $\Omega_{\epsilon}^c$ the implication
\begin{align}
  \max_{l = 1, \ldots, m}\sup_{s\in [t/2, t]} |\langle \JJ_{s,t}^* \eta, \dir^{(l)} \rangle| \leq \epsilon^{\delta}
  \quad \Rightarrow  \quad
  \max_{\substack{l = 1,\ldots, m\\ j = 1, \ldots,d}} \sup_{s\in [t/2, t]} |\langle \JJ_{s,t}^* \eta, [[F,\dir^{(l)}],\sigma_j]\rangle| \leq \epsilon^{\frac{\delta}{80}}
  \label{eq:implication:for:iteration}
\end{align}
holds for every $\eta \in \Dtest$.
\end{Lem}
\begin{proof}
The proof proceeds in two steps.  We first find sets $\Omega_{\epsilon,1}$
for which the analogue of \eqref{eq:good:set:cond} holds and where, off of $\Omega_{\epsilon,1}$,
\begin{align}
  \max_{l = 1, \ldots, m}\sup_{s\in [t/2, t]} |\langle \JJ_{s,t}^* \eta, \dir^{(l)} \rangle| \leq \epsilon^{\delta}
  \quad \Rightarrow  \quad
  \max_{\substack{l = 1,\ldots, m\\ j = 1, \ldots,d}} \sup_{s\in [t/2, t]} |\langle \JJ_{s,t}^* \eta, [F,\dir^{(l)}]\rangle| 
  \leq \epsilon^{\frac{\delta}{5}}.
  \label{eq:implication:for:iteration:step1}
\end{align}
The second step is to find $\Omega_{\epsilon,2}$, again satisfying bounds as in \eqref{eq:good:set:cond} 
so that outside of these exceptional sets
\begin{align}
   \max_{\substack{l = 1,\ldots, m\\ j = 1, \ldots,d}} \sup_{s\in [t/2, t]} |\langle \JJ_{s,t}^* \eta, [F,\dir^{(l)}]\rangle| 
  \leq \epsilon^{\frac{\delta}{5}}
  \quad \Rightarrow  \quad
  \max_{\substack{l = 1,\ldots, m\\ j = 1, \ldots,d}} 
  \sup_{s\in [t/2, t]} |\langle \JJ_{s,t}^* \eta, [[F,\dir^{(l)}],\sigma_j]\rangle| \leq \epsilon^{\frac{\delta}{80}}.
  \label{eq:implication:for:iteration:step2}
\end{align}
Combining the sets by forming $\Omega_\epsilon =  \Omega_{\epsilon,1} \cup \Omega_{\epsilon, 2}$
yields the desired result.

The first step relies on another application of Lemma~\ref{lem:small:to:small:prime:wrapper}.  Define 
\begin{align*}
  g_\eta^{(l)}(s) = \langle \JJ_{s,t}^* \eta, \dir^{(l)} \rangle,  \quad \eta \in \Dtest,
\end{align*}
and observe that
\begin{align*}
  \frac{d}{ds}g_\eta^{(l)}(s)  = \langle \frac{d}{ds}\JJ_{s,t}^* \eta, \dir^{(l)} \rangle=\langle \JJ_{s,t}^* \eta, [F,\dir^{(l)}] \rangle. 
\end{align*}
In order to apply \eqref{eq:exceptional:sets:est:derivatives} we estimate 
\begin{align}
  \E  \left( \sup_{\eta \in \Dtest} \| \langle \JJ_{s,t}^* \eta, [F,\dir^{(l)}] \rangle \|_{C^{1/4}}^{8q } \right)
   \leq  C \left[  \E  \left( \sup_{\eta \in \Dtest} \| \JJ_{s,t}^* \eta \|_{C^{1/4}}^{16q } \right) \E \left(\| [F,\dir^{(l)}] \|_{C^{1/4}}^{16q} \right)
   \right]^{1/2}.
     \label{eq:careful:brother:rat}
\end{align}
To address the first term, referring to \eqref{eq:J:op:def}, we find
\begin{align}
  \| \JJ_{s,t}^* \eta \|_{C^{1/4}}^{16q }  
  \leq& \sup_{s \in [0, t]} \| \JJ_{s,t} \|^{16q} 
  + \sup_{s \in [0,t]} |\nu A\JJ_{s,t}^*\eta + (\nabla B(U))^* \JJ_{s,t}^*\eta|^{16q} 
  \notag\\
    \leq& C\left(1 + \sup_{s \in [0, t]} \| \JJ_{s,t} \|^{32q} + \sup_{s \in [0,t]} |U(s)|^{32q}\right).
    \label{eq:U:quant:est:2}
\end{align}
On the other hand, observing that
\begin{align}
   [F,\dir^{(l)}] = \nu A\dir^{(l)} + B(U, \dir^{(l)}) + B(\dir^{(l)},U)
   \label{eq:first:brack:explicit}
\end{align}
we have that $\| [F,\dir^{(l)}] \|_{C^{1/4}}^{16q} \leq C (\| U\|_{C^{1/4}}^{16q} + 1)$.  Now
\begin{align}  
   \|U\|_{C^{1/4}}^{16q} &\leq 
   C\left(\sup_{s \in [0,t]} |U|^{16q} + \sup_{s \in [0,t]} |\nu AU + B(U,U)|^{16q} + \|\sigma W\|_{C^{1/4}}^{16q}\right) \notag\\
   &\leq
   C\left(1 + \sup_{s \in [0,t]} |U(s)|^{32q} + \|\sigma W\|_{C^{1/4}}^{16q}\right).
   \label{eq:U:quant:est:3}
\end{align}
Combining \eqref{eq:U:quant:est:2}, \eqref{eq:U:quant:est:3}, making use of the 
regularity properties of Brownian motion\footnote{A convenient way to estimate
$\E \|\sigma W\|_{C^{1/4}}$ is to use a fractional (in time) Sobolev embedding
argument; see \cite{ZabczykDaPrato1992}.} and \eqref{eq:exp:martingale:lead:to:awesome:2} 
we finally conclude that
\begin{align*}
 \E  \left( \sup_{\eta \in \Dtest} \| \langle \JJ_{s,t}^* \eta, [F,\dir^{(l)}] \rangle \|_{C^{1/4}}^{8q } \right)
 \leq C \E\left(1 + \sup_{s \in [0, t]} \| \JJ_{s,t} \|^{32q} + \sup_{s \in [0,t]} |U(s)|^{32q}\right)
 \leq C \exp(\expCon |U_0|^2).
\end{align*}
We therefore invoke Lemma~\ref{lem:small:to:small:prime:wrapper} and infer the existence of sets $\Omega_{\epsilon,1} = \cup_{l = 1, \ldots, m} \Omega_{\epsilon,1, l}$ with $\Prb( \Omega_{\epsilon,1}) \leq C \epsilon^q$ such that off of $\Omega_{\epsilon,1}$, \eqref{eq:implication:for:iteration:step1} holds.

We turn now to finding sets on which the second implication \eqref{eq:implication:for:iteration:step2} holds. 
Working from \eqref{eq:first:brack:explicit} and recalling that $U = \bar{U} + \sum_{j = 1}^d \sigma_j W^j$,
where $\bar{U}$ solves \eqref{eq:shifted:equation:bm},
we have
\begin{align}
   \langle \JJ_{s,t}^* \eta, [F,\dir^{(l)}] \rangle
   =   \langle \JJ_{s,t}^* \eta,  \nu A\dir^{(l)} + B(\bar{U}, \dir^{(l)}) + B(\dir^{(l)},\bar{U})\rangle 
        + \sum_{j =1}^d \langle \JJ_{s,t}^* \eta, [[F,\dir^{(l)}], \sigma_j] \rangle W^j.
        \label{eq:our:big:wiener}
\end{align}
For $\epsilon > 0$, take $\tilde{\Omega}_{\epsilon}$ to be the exceptional sets corresponding to the implication 
\eqref{eq:good:wiener:behavior} obtained in Lemma~\ref{lem:wiener:poly:simple:case} for $r = \delta/5$.
Now define $\Omega_{\epsilon,2} = \tilde{\Omega}_{\epsilon} \cup \bigcup_{j= 0}^{d} \tilde{\Omega}_{\epsilon,j}$ 
where\footnote{Note here that we needed that the implication \eqref{eq:good:wiener:behavior}
to hold universally for all Wiener polynomials on $\tilde{\Omega}_{\epsilon}$ since otherwise we would be 
intersecting a countable number of sets.}
\begin{align*}
   \tilde{\Omega}_{\epsilon,0} 
  := \left\{ \sup_{\eta \in \Dtest}\sup_{s \in [0,t]} \left| \frac{d}{ds} \langle \JJ_{s,t}^* \eta,  \nu A\dir^{(l)} 
  + B(\bar{U}, \dir^{(l)}) + B(\dir^{(l)},\bar{U})\rangle \right| \geq \epsilon^{- \delta / 80} \right\}
\end{align*}
and
\begin{align*}  
   \tilde{\Omega}_{\epsilon,j} := \left\{ \sup_{\eta \in \Dtest}\sup_{s \in [0,t]}
  \left|\frac{d}{ds} \langle \JJ_{s,t}^* \eta, [[F,\dir^{(l)}], \sigma_j] \rangle \right| \geq \epsilon^{- \delta / 80} \right\}.
\end{align*}
Notice that the second desired implication \eqref{eq:implication:for:iteration:step2} holds over all $\eta \in \Dtest$
off of $\Omega_{\epsilon,2}$.  Thus, the proof will be complete once we find suitable a bound for this set. 

Observe that
\begin{align*}
   \Prb( {\Omega}_{\epsilon,0} )
   \leq C\epsilon^{q} \!
   \left( \E \left(\sup_{\eta \in \Dtest}\sup_{s \in [0,t]}\|\tfrac{d}{ds} \JJ_{s,t}^* \eta\|^{2p} + \sup_{s \in [0,t]}\| \JJ_{s,t}\|^{2p} \right)
   	\cdot \E \left(\sup_{s \in [0,t]}(1 + |U|^{4p} + |\sigma W|^{2p}) \right) \right)^{1/2}
   \! \! \leq C \epsilon^{q},
\end{align*}
for some sufficiently large $p = p(\delta)$ and where
$C = C(\nu, |\sigma|, |U_{0}|)$ is independent of $\epsilon$.  
Here we have used \eqref{eq:J:bnd} and bounds as in \eqref{eq:U:quant:est:2} with \eqref{eq:exp:martingale:lead:to:awesome:2} to achieve the final
inequality.  For $\tilde{\Omega}_{\epsilon,j}$ we have
\begin{align*}
	 \Prb( {\Omega}_{\epsilon,j} )    \leq C\epsilon^{q} \E \left(\sup_{\eta \in \Dtest}\sup_{s \in [0,t]}\|\tfrac{d}{ds} \JJ_{s,t}^*\|^{p}\right)
	 \leq C\epsilon^q.
\end{align*}
Combining these bounds therefore completes the proof of Lemma~\ref{lem:direction:to:F:bomb}.

\end{proof}
\begin{Rmk}\label{rmk:regularity:issue}
  As mentioned above an estimate very similar to \eqref{eq:Mall:spec:est} may be obtained for 
  certain infinite dimensional problems.   While the approach is very 
  similar one need to take considerable care with the regularity of solutions in order to obtain
  bounds analogous to e.g. \eqref{eq:U:quant:est:1} or \eqref{eq:U:quant:est:2}.  Here one needs
  to establish higher regularity (in space) bounds, to consider more restrictive time intervals in Lemmas~\ref{lem:inital:step},
  \ref{lem:direction:to:F:bomb} or both.
 See \cite{HairerMattingly2011, FoldesGlattHoltzRichardsThomann2013} for further details.
\end{Rmk}

\subsection{Combining the Chain of Implications (proof of Proposition~\ref{prop:main:spec:bound}, conclusion)}

Fix any $q > 0$ and let $n$ be such that $\mbox{span} (\mathcal{W}_n) = \RR^N$.   Pick a countable
dense subset $\Dtest$ of $\SN$.  
We apply Lemma~\ref{lem:inital:step} to determine
a set $\Omega_{\epsilon, 0}$
off of which \eqref{eq:implication:initial:step} holds over $\eta \in \Dtest$.  Next, for each $k = 1, \ldots, n$, 
we apply Lemma~\ref{lem:direction:to:F:bomb}
with $\delta = (1/80)^{k-1}$ and where the elements
$\xi^{(l)}$ consist of the elements in $\mathcal{W}_{k-1}$ defined by \eqref{eq:simplified:hor}.
We thus obtain sets $\Omega_{\epsilon, k}$
such that off these sets the implication \eqref{eq:implication:for:iteration} holds for $\eta \in \Dtest$.
Now take $\Omega_{\epsilon} = \cup_{k = 0}^n \Omega_{\epsilon, k}$ and observe that
\begin{align*}
  \Prb(\Omega_\epsilon) \leq C \epsilon^q
\end{align*}
for some $C = C(\nu, |U_0|)$, independent of $\epsilon$ and, off of this exceptional set,
\begin{align}
   \langle \MM_{0,t} \eta, \eta \rangle  \leq \epsilon
   \quad \Rightarrow \quad  \sum_{\xi \in \mathcal{W}_n} \langle \eta, \xi \rangle^2 \leq C\epsilon^\gamma
   \label{eq:total:imp}
\end{align}
for some $C$, $\gamma$ independent of $\epsilon$, and holds for every $\eta \in \Dtest$.   On the other hand, since
$\mathcal{W}_n$ is assumed to be a spanning set, there exist $\kappa > 0$ such that
\begin{align}
  \sum_{\xi \in \mathcal{W}_n} \langle \eta, \xi \rangle^2 \geq \kappa, \textrm{ for every } \eta \in \SN.
  \label{eq:span:conclus}
\end{align}
Combining \eqref{eq:total:imp} with \eqref{eq:span:conclus} we conclude that
\begin{align*}
   G(\eta) := \frac{\langle \MM_{0,t} \eta, \eta \rangle}{|\eta|^2} \geq \epsilon
\end{align*}
for every $\eta$ on a dense subset $\RR^N$ off of $\Omega_\epsilon$, for all $\epsilon$ sufficiently small.  
Since $G$ is continuous in $\eta$ the desired conclusion \eqref{eq:Mall:spec:est} now follows along with the proof
of Proposition~\ref{prop:main:spec:bound} and hence of Proposition~\ref{prop:strong:feller:cond}.

\section*{Acknowledgments}
These notes were written as a supplement to lectures
given in as a mini course at the Thematic Program on Incompressible Fluid Dynamics
at Instituto Nacional de Matem\'atica Pura e Aplicada (IMPA) in Rio de Janeiro, Brazil.
I am grateful to Helena Nussenzveig Lopes and Milton Lopes Filho for the invitation to give these lectures.  
This work was partially supported by the National Science Foundation under the grant
NSF-DMS-1313272.

\addcontentsline{toc}{section}{References}
\begin{footnotesize}
\bibliographystyle{alpha}
\bibliography{bib}
\end{footnotesize}

\vspace{.3in}
\noindent
Nathan Glatt-Holtz\\ {\footnotesize
Department of Mathematics\\
Virginia Polytechnic Institute and State University\\
Web: \url{www.math.vt.edu/people/negh/}\\
 Email: \url{negh@vt.edu}} \\[.2cm]

\end{document}